\documentclass[11pt]{amsart}
\usepackage{amsmath,amscd,amssymb}
\usepackage[mathscr]{eucal}

\title[On regularity of reciprocals for exponential series]
{On the problem by Erd\H{o}s-de Bruijn-Kingman on regularity of reciprocals for exponential series}

\newcommand{\vanish}[1]{\relax}

\newcommand{\N}{\mathbb{N}}

\newcommand{\R}{\mathbb{R}}
\newcommand{\C}{\mathbb{C}}

\newcommand{\T}{\mathbb{T}}
\newcommand{\D}{\mathbb{D}}

\newcommand{\Sum}[2][\relax]{%
 \ifx#1\relax \sideset{}{_{#2}}\sum
 \else \sideset{}{^{#1}_{#2}}\sum
 \fi}

\newcommand{\norm}[2][\relax]{%
   \ifx#1\relax \ensuremath{\left\Vert#2\right\Vert}
   \else \ensuremath{\left\Vert#2\right\Vert_{#1}}
   \fi}

\makeatletter
\newcommand{\sprod}[2]{\ensuremath{%
  \setbox0=\hbox{\ensuremath{#2}}
  \dimen@\ht0
  \advance\dimen@ by \dp0
  \left(\left.#1\rule[-\dp0]{0pt}{\dimen@}\,\right|#2\hspace{1pt}\right)}}
 \makeatother

\newcounter{aufzi}

\newcounter{aufzii}

\newcounter{aufziii}

 \newtheorem{thm}{Theorem}[section]
 \newtheorem{cor}[thm]{Corollary}
 \newtheorem{lemma}[thm]{Lemma}
 \newtheorem{prop}[thm]{Proposition}

 \theoremstyle{definition}

 \theoremstyle{remark}
 \newtheorem{rem}[thm]{Remark}

\newtheorem{example}[thm]{Example}

\newtheorem{remark}[thm]{Remark}

\numberwithin{equation}{section}

\begin{document}

\author{Alexander Gomilko}
\address{Faculty of Mathematics and Computer Science\\
Nicolaus Copernicus University\\
ul. Chopina 12/18\\
87-100 Toru\'n, Poland\\
and Institute of Telecommunications and Global \\
Information Space, National Academy of Sciences of Ukraine\\
Kiev, Ukraine
}

\email{gomilko@mat.umk.pl}

\author{Yuri Tomilov}
\address{Institute of Mathematics\\
Polish Academy of Sciences\\
\'Sniadeckich 8\\
00-956 Warsaw, Poland
}

\email{ytomilov@impan.pl}





\begin{abstract}
Motivated by applications to renewal theory,  Erd\H{o}s, de Bruijn and Kingman posed a problem on boundedness
of reciprocals $(1-z)/(1-F(z))$ in the unit disc for probability generating functions $F(z)$. It was solved by Ibragimov in $1975$ by constructing a counterexample.
In this paper, we  provide much stronger counterexamples showing that the problem does not
allow for a positive answer even under rather restrictive additional assumptions.
Moreover, we pursue a systematic study of $L^p$-integrabilty properties for the reciprocals.
In particular, we show that while the boundedness of $(1-z)/(1-F(z))$ fails in general,
the reciprocals do possess certain $L^p$-integrability properties under mild conditions  on $F$.
We also study the same circle of problems in the continuous-time setting.
\end{abstract}

\maketitle
\section{Introduction}
\subsection{Motivation}
The paper addresses several notorious  problems related to renewal sequences and their generating functions.
Recall that if  $(a_k)_{k=1}^{\infty}$ is such that $a_k \ge 0, k \ge 1,$ and $\sum_{k \ge 1} a_k=1,$
then the sequence $(b_k)_{k=0}^{\infty}$ given by the recurrence relation
\begin{equation}\label{recrel}
b_n=\sum_{k=1}^{n} a_k b_{n-k}, \qquad b_0=1, \qquad n \in \mathbb N,
\end{equation}
is called the renewal sequence associated to $(a_k)_{k=1}^{\infty}.$
Renewal sequences are a classical subject of studies in probability theory, in particular, in the theory of Markov processes.
To mention one of the probabilistic meanings of \eqref{recrel},
  let $\{X_n: n \ge 0\}$ be a recurrent Markov chain
 with the state space $\mathbb N \cup \{0\}$ and  $P(X_0 = 0) = 1$. If $T:=\inf \{n \ge 1:X_n=0\}$ denotes  the
time of first return to the origin, and $a_n:=P(T=n), n \in \mathbb N,$ 
then $b_n=P(X_n=0|X_0=0), n \in \mathbb N.$
For a thorough discussion of probabilistic background for \eqref{recrel}, see e.g. \cite[Vol 1, Ch. XIII]{Feller} and \cite[Ch. 1]{Kingman}.

Moreover, renewal sequences are also of substantial interest  in ergodic theory. For applications in ergodic theory one may consult e.g.
the papers \cite{Aa}, \cite{Aa1} and \cite{MRL}, the book \cite{AaB} and the references therein.

It is often convenient to study $(a_k)_{k=1}^{\infty}$ and $(b_k)_{k=0}^{\infty}$ in terms of their generating functions $F$ and $G,$ given by
$$
F(z)=\sum_{k=1}^{\infty} a_k z^k \qquad \text{and} \qquad G(z)=\sum_{k=0}^{\infty} b_k z^k.
$$
The functions are defined on the open unit disc
$
\D:=\{z\in \C: |z|<1\}
$
and connected by the relation
$$
G(z)=\frac{1}{1-F(z)}.
$$

Being unable to give any account of the wide topic of renewal sequences we refer to the classical
sources such as for instance \cite{Kingman}, \cite{Kendall}, \cite{Spitzer}, and \cite{Feller}
(although the term ``renewal sequence'' for $(b_n)_{n=0}^{\infty}$ given by
\eqref{recrel} is used only in \cite{Kingman} and \cite{Kendall}).
\subsection{History}
One of the first and foundational results in theory of renewal sequences is the famous Erd\H{o}s-Feller-Pollard theorem. To recall it
we need to introduce certain notation.
Let $\mathcal{A}^+$ consist
of the power series of the form
\begin{equation}\label{Ber}
F(z)=\sum_{k=1}^\infty a_k z^k,\qquad a_k\ge 0,\qquad \sum_{k=1}^\infty a_k=1
\end{equation}
in  $\overline{\D}.$ It is a complete metric space with metric induced
by $\ell_1$-norm on an appropriate sequence space.
We say that $F \in \mathcal A^+$ is aperiodic if $F(z)=1, z \in  \overline{\D},$ implies that $z=1.$
Clearly, if $F$ is aperiodic then $1/(1-F)$ is analytic in $\mathbb D$ and continuous in ${\mathbb D}\setminus\{1\}.$

Using Wiener's theorem, it
 was proved in \cite{Erdos} that if $F\in \mathcal{A}^+$
is aperiodic and additionally
\begin{equation}\label{erdos}
\mu:=\sum_{k=1}^\infty k a_k<\infty,
\end{equation}
then
\begin{equation}\label{limit}
\lim_{k \to \infty} b_k=1/\mu.
\end{equation}
This is essentially the famous  Erd\H{o}s-Feller-Pollard theorem, one of the first and basic limit theorems in renewal theory.

The key point in \cite{Erdos} for showing  the property \eqref{limit} was the fact that the function
$(1-z)(1-F)^{-1}$
has absolutely convergent Taylor series:
\begin{equation}\label{BEQ}
\sum_{k=1}^{\infty} |b_k-b_{k+1}|<\infty.
\end{equation}

The theorem generated an area of research, and a huge number of its generalizations and improvements  in various directions
has appeared in subsequent years.
Analytic approaches to the study of $1/(1-F)$ and of asymptotics of $(b_k)_{k=1}^{\infty}$ are discussed e.g. in \cite[Chapter V.22]{Korevaar} and \cite[Chapter 24]{Postnikov}.
These books contain a number of related references. We mention here only the classical papers \cite{Stone} and \cite{Ericsson}.

However, certain natural questions have escaped a thorough study.
In particular, P. Erd\H{o}s and N. de Bruijn suggested in \cite[p. 164]{Bruijn}
that (\ref{BEQ})
is probably true for any aperiodic $F\in \mathcal{A}^+$ and the assumption \eqref{erdos} is redundant.
As they wrote in \cite{Bruijn}, ``it seems possible that the condition \eqref{erdos} is
superfluous''.
Moreover, the question whether \eqref{BEQ} holds for any aperiodic $F$ satisfying \eqref{Ber}
was formulated as an open problem by J. Kingman in \cite[p. 20-21, (iv)]{Kingman}.
A recent discussion of the problem in the context of ergodic theory  can be found in \cite{Aa1}.
The analysis of $(1-z)(1-F)^{-1}$ presents certain difficulties in view of nonlinear character
of the transformation $F \mapsto (1-F)^{-1}.$
While $(b_k)_{k=0}^{\infty}$ is given explicitly in terms of $(a_k)_{k=1}^{\infty},$ it is very difficult
to study it by means of the recurrence relation \eqref{recrel}
(see e.g. \cite{Bruijn1} and \cite{Bruijn} for such a direct approach).
So most of the research on analytic properties of renewal sequences concentrated on the generating-function methodology.

One must note that relevant studies have been made by  J. Littlewood
in \cite{Littl}, a paper apparently overlooked by mathematical community.
Being motivated by the enigmatic message from Besicovitch (see \cite[p. 145]{Littlewood}) and a question by W. L. Smith, Littlewood proved in \cite{Littl} that
for any function $f$ given by
\begin{equation}\label{lambda}
f(\theta)=\sum_{k=1}^\infty a_k e^{i\lambda_k\theta},\qquad \text{where} \,\, \lambda_k \in \mathbb R, \,\, \lambda_k\ge 1, \,\, \lim_{k\to\infty}\,\lambda_k=\infty,
\end{equation}
and $(a_k)_{k=1}^{\infty}$ as in \eqref{Ber},
one has
\begin{equation}\label{Lpart}
\overline{\lim_{\theta_0\to 0}}\,
\int_{\theta_0}^{2\theta_0}
\frac{d\theta}{|1-f(e^{i\theta})|}<\infty.
\end{equation}
(Sometimes  $f$ satisfying \eqref{lambda} are called quasi-exponential series.)
In particular, there is $\delta>0$ (depending on $f$) such that
\begin{equation}\label{aper}
\int_0^\delta
\frac{\theta^\alpha d\theta}{|1-f(e^{i\theta})|}<\infty
\end{equation}
for any $\alpha>0.$
Results of that type lead to a number of useful consequences in the study of regularity for generating functions of renewal sequences,
as we show in Section \ref{lp}.

It is natural to ask whether Littlewood's results can be essentially improved. For example, boundedness from below of
$ |1-f(e^{i\theta})|/\theta$ in the neighborhood of zero would imply
 \eqref{Lpart}. Littlewood's student H. T. Croft claimed in \cite{Croft} that the latter property does not hold, in general.
More precisely, if $f$ is defined by \eqref{lambda}, then for any function $\chi$
such that $\chi(\theta) \uparrow \infty$
as $\theta\to 0$
there exist sequences $(a_k)_{k=1}^{\infty}$ and
$(\lambda_k)_{k=1}^{\infty}$ as above, and
$(\theta_k)_{k=1}^{\infty}$
satisfying $\theta_k\to 0$ as $k\to\infty$ such that
\begin{equation}\label{CroftI}
|1-f(\theta_k)|\le \chi(\theta_k)\theta_k^2,\qquad
k\in\N.
\end{equation}
(In fact, only the case $\chi(\theta)=\theta^{-\epsilon}$ was discussed in \cite{Croft}.)
This, indirectly, would solve the Erd\H{o}s-de Bruijn-Kingman problem once one would arrange the integer frequencies $\lambda_k$ above, although Croft presumably was not aware of the problem.
However, \cite{Croft} contains only a hint rather than a complete argument, and it produces  merely  real frequencies $\lambda_k$ rather than integer ones as in \eqref{Ber}.

The  Erd\H{o}s-de Bruijn-Kingman problem was settled in the negative by I. A. Ibragimov  in \cite{Ibragimov},
who constructed  an $F$ such that $(1-z)/(1-F(z))$ is unbounded in $\D.$
However, the size of the gap between the Erd\H{o}s-Feller-Pollard condition \eqref{erdos}
and the situation with no a priori assumption, i.e. $\sum_{k\ge1} a_k=1,$  remained completely
unclear. In this paper, we show that \eqref{BEQ} does not in general hold under essentially any summability
assumptions weaker than \eqref{erdos},
and moreover \eqref{BEQ} fails for  generic probabilities $(a_k)_{k=1}^{\infty}.$

It is also instructive to remark that in \cite{Hawkes}
J. Hawkes  constructed a lacunary series of the form (\ref{lambda}) with $\theta_k=2^{-2^{k^2}}$ and $\lambda_k=\frac{2\pi-\theta_k^{2/3}}{\theta_k}, \, k\in \N,$ such that
\[
\lim_{k\to \infty}\,
\frac{|1-f(\theta_k)|}{\theta_k}=0.
\]
This way, Hawkes solved another Kingman problem formulated in \cite[p. 76]{Kingman}, which is similar (but not equivalent) to the problem mentioned above. However, the approaches of \cite{Hawkes} and \cite{Kingman} are quite close to each other.

\subsection{Results}

In this paper, we revisit the problem posed by Erd\H{o}s, de Bruijn and Kingman,
and provide counterexamples that can be considered as, in a sense, best possible.
Namely, by methods very different from those of \cite{Ibragimov},
we prove in
Theorem \ref{betakN}
that for any positive sequence
$(\epsilon_k)_{k=1}^{\infty}$ tending to zero
(subject to a technical assumption)
there exists an aperiodic
$F\in \mathcal{A}^+$ with
\begin{equation}\label{kbetak}
\sum_{k=1}^\infty k \epsilon_k  a_k<\infty,
\end{equation}
such that
$
(1-z)(1-F)^{-1}$
is not even bounded in $\mathbb D$,
and thus
(\ref{BEQ}) is not true. Moreover, the set of such $F$ is dense
in $ \mathcal A^+$ (when $\mathcal A^+$ is considered as a metric space with a natural metric).
Thus, the assumption \eqref{erdos} in the Erd\H{o}s-Feller-Pollard theorem is optimal
as far as the ``smoothness'' of $(b_k)_{k=0}^{\infty}$ is concerned.
Several results of a similar nature have been obtained as well.
At the same time, we show in Appendix B that Croft's idea can successfully be realized, and moreover it can also be realized for
integer frequencies.

Our technique is based on constructing special sequences of polynomials approximating a given polynomial well enough in an appropriate norm
and, as in \eqref{CroftI}, the constant function $1$ at a sequence of points from the unit circle converging to $1.$ By means of either Baire-category arguments or inductive reasoning, this then turns into the same estimates
for exponential series:
\[
f(\theta)=F(e^{i\theta}),\qquad |\theta|\le \pi,\qquad
F\in \mathcal{A}^+.
\]
It is crucial that the bounds  of the type \eqref{CroftI} can also be spread out to an appropriate sequence of intervals approaching  $1,$ and
thus hold on a set of sufficiently large
measure. These extended bounds generalize the upper estimates
from  \cite{Croft}, \cite{Ibragimov} and
\cite{Hawkes}, and they allow us
to get rid of a certain amount of regularity of $(1-F)^{-1},$ e.g. with respect to the $L^p$-scale.

By pursuing our studies a bit further, it is natural to ask what kind regularity is possessed by $(1-F)^{-1}$ \emph{without any a priori assumptions}
on the sequence of Taylor coefficients $(a_k)_{k=1}^{\infty}$ of $F.$
Despite the enormous number of papers on renewal sequences, the question seems to have not been adequately
addressed so far
(apart perhaps to some extent \cite{Aa}, \cite{Aa1} and \cite{Littl})
In the present paper, we make several steps in this direction.
First, we extend Littlewood's results (\ref{Lpart})
and (\ref{aper}) by relating the integrability of $(1-f)^{-1}$ on an interval $(\theta_0,2\theta_0)\subset (0,2)$ to the summability properties of the Taylor coefficients
of $F.$ This allows us to obtain sharp and explicit conditions for the integrability of  $(1-F)^{-1}$ on the unit circle $\mathbb T$ 
if $F$ is aperiodic.
Furthermore, we pursue a similar study for the ``smoothed'' function $(1-z)(I-F)^{-1}$ appearing in the Erd\H{o}s-de Brujin-Kingman problem.
We show that for $F$ as in \eqref{Ber}, satisfying
\begin{equation}\label{nu1}
\sum_{k=1}^\infty k^\nu a_k<\infty
\end{equation}
for some $\nu \in (0,1),$ one has
$$(1-z)(1-F)^{-1} \in L^{1+\frac{1}{1-\nu}}(\mathbb T).$$
On the other hand,
for each $p \in (2+(1-\nu)^{-1},\infty)$ we construct a function $F_p$ of the form \eqref{Ber} satisfying
\eqref{nu1} but at the same time violating
 $$(1-z)(1-F_p)^{-1} \not\in L^p(\mathbb T).$$

 Remark that while \eqref{BEQ} is not, in general, true for $F \in \mathcal A^+$ (as we show in this paper),
 we prove that nevertheless a weaker property holds:
 \begin{equation*}
\sum_{k=0}^{\infty} (b_k-b_{k+1})^2<\infty.
\end{equation*}
This simple result has probably been overlooked in the literature.
Moreover, we show that, in general,  $(1-z)(1-F)^{-1} \not \in L^p(\mathbb T)$
if $p \in (3,\infty).$ The problem what happens if  $p \in (2,3]$ remains, unfortunately, open.

We finish the paper with remarks on a continuous analogue of the Erd\H{o}s-de Bruijn-Kingman problem.

\section{Preliminaries and notations}\label{prelim}

For $ w=(w_k)_{k=1}^{\infty} \subset [1,\infty)$ we denote by
$\mathcal{A}(\omega)$
a Banach space
\[
\left\{f(\theta)=\sum_{k=1}^\infty
a_k e^{i k\theta}:\;
\; \sum_{k=1}^\infty w_k |a_k| <\infty,\quad  |\theta|\le \pi
\right\},
\]
with the norm
\[
\| f\|_{\mathcal A(w)}
=\sum_{k=1}^\infty w_k|a_k|, \qquad f \in \mathcal A(w).
\]
Its subset $\mathcal{A}^+(w)$
given by
\[
\mathcal{A}^+(w)=\left\{\sum_{k=1}^\infty
a_k e^{i k\theta} \in \mathcal {A}(\omega):\;
\;a_k \ge 0, \, \,\, \sum_{k=1}^\infty a_k=1
\right\}
\]
is a complete metric space with the metric $\rho(\cdot, \cdot)_{\mathcal {A}(w)}$ inherited from $\mathcal{A} (w).$

Note that
\begin{equation}\label{Ina2}
|f(\theta)-g(\theta)|\le \|f-g\|_{\mathcal A(w)},\qquad f,g\in \mathcal{A}^+(w),\qquad |\theta|\le \pi.
\end{equation}

We will often use a more intuitive notation
$\|f_1-f_2\|_{\mathcal A(w)}$ instead of $\rho(f_1, f_2)_{\mathcal A(w)}$ whenever it is defined correctly.
If $w_k = k^\nu, \nu \in [0,1),$ for $k \in \mathbb N,$ then we will write $\mathcal A(\nu)$ instead of $\mathcal A(w)$ slightly abusing our notation. The same will apply to the spaces $\mathcal A^+(w).$
We will also write $\mathcal A^+$ (respectively $\mathcal A$) instead of $\mathcal A^+(\{1\})$ (respectively $\mathcal A(\{1\})$).
Clearly, $\mathcal A$ is embedded contractively into any $A(w)$ with $w$ as above.
It will sometimes be convenient to consider functions from $A(w)$ extended periodically
to the whole real line.

In the sequel,  we identify absolutely convergent power series $F(z)=\sum_{k=1}^n a_k z^k$ on $\overline{\mathbb D}$ with their boundary values on $\mathbb T,$
and the boundary values with the corresponding $2\pi$-periodic functions $f,$ so that
\begin{equation}\label{IsForm}
f(\theta):=F(e^{i\theta})
=\sum_{k=1}^\infty a_k e^{ik\theta},\qquad |\theta|\le \pi.
\end{equation}
In this way, our notation $\mathcal A^+$ agrees with the same notation used in the introduction for power series.
A function $f \in \mathcal{A}^+$ of the form $f=\sum_{k=1}^n a_k e^{i k\theta}, n \in \N,$ will be called
(exponential) polynomial. As usual, its degree ${\rm deg}\, f$ is defined as ${\rm deg}\, f:=\max\{k\in \mathbb N: a_k \neq 0\}.$
Clearly the set of polynomials is contained in any $\mathcal A(w).$ 

Following the definition of aperiodic $F$ from the introduction, $f \in \mathcal{A}^+$ is said to be aperiodic if $f(\theta)=1$ for some $\theta \in [-\pi,\pi],$ implies $\theta=0.$
Let us recall that $f(\theta)=\sum_{k=1}^\infty
a_k e^{i k\theta} \in \mathcal{A}^+$ is aperiodic if and only if the greatest common divisor of $\{k \in \mathbb N: a_k > 0\}$ is $1.$
The argument for the ``only if'' part of this equivalence can be found e.g. in \cite[p. 12]{Kingman} or \cite[p. 272-273]{Korevaar},
while the other part is obvious. The equivalence, in particular, implies that 
 if $a_1\not=0,$  then 
$f$ is aperiodic.
Moreover, if $f$ is not aperiodic, then there exists $\theta \in [\pi/2,\pi]$ such that  $f(\theta)=1.$

Observe that the set of polynomials in $\mathcal A^+(w)$ is dense in $\mathcal A^+(w)$ for any weight $w.$ Indeed,
let $f \in \mathcal A^+(w)$
be given by
\begin{equation}\label{formAP}
f(\theta)=\sum_{k=m}^\infty a_k e^{ik\theta},\qquad a_m\not=0.
\end{equation}
Let us define for $n\ge m+1$ the family
of aperiodic polynomials
\[
P_n(\theta)=\left(\frac{a_m}{n}e^{i\theta}+
a_m\left(1-\frac{1}{n}\right)e^{i m\theta}+
\sum_{k=m+1}^n a_k e^{i k\theta}\right)/d_n\in \mathcal{A}^{+},
\]
where
\[
d_n=\sum_{k=m}^n a_k \to 1,\quad n\to\infty.
\]
Then
\begin{align*}
\|f-P_n\|_{\mathcal A(w)}
\le& (1/d_n-1)\|f\|_{\mathcal A(w)}+d_n^{-1}\|f-d_nP_n\|_{\mathcal A(w)}\\
\le& (1/d_n-1)\|f\|_{\mathcal A(w)}+
\frac{(w_1+w_m) a_m}{n d_n}\\
+&\frac{1}{d_n}\sum_{k=n+1}^\infty w_k a_k\to 0,\quad n\to\infty.
\end{align*}

The next simple proposition will be useful for the sequel. It is probably known, but we were  not able
to find an appropriate reference.

\begin{prop}\label{Open}
The set of aperiodic functions in  $\mathcal A^{+}(w)$ is open
in $\mathcal A^{+}(w)$.
\end{prop}

\begin{proof}
Let $(f_n)_{n=1}^{\infty}\subset \mathcal A^{+}(w)$ be a sequence of non-aperiodic functions
such that
\begin{equation}\label{limAp}
\lim_{n\to\infty}\,\|f_0-f_n\|_{\mathcal A(w)}=0
\end{equation}
for some $f_0 \in A^{+}(w).$ 
Note that for every $n\in \mathbb \N$ there exists $\theta_n\in [\pi/2,\pi]$ such that
$
f_n(\theta_n)=1.
$
If $\theta_0$ is any limit point of $(\theta_{n})_{n=1}^{\infty},$ then
$\theta_0\in [\pi/2,\pi],$ and from (\ref{limAp}),
(\ref{Ina2}) and the continuity of $f_0$ it follows that $f_0(\theta_0)=1$.
Therefore,  $f_0$ is not aperiodic, and the set of non-aperiodic functions is closed in $\mathcal A^{+}(w)$.
\end{proof}

\begin{rem}\label{aperf}
By Proposition \ref{Open} the set of aperiodic functions
in $\mathcal A^{+}(w)$ is open in $\mathcal A^{+}(w).$
Since that set is also dense in $\mathcal A^{+}(w)$ as we showed above,
the set of aperiodic functions in
 $\mathcal A^{+}(w)$ is residual, i.e. it is the complement of a set of first category in $\mathcal A^{+}(w).$
\end{rem}

Finally, we will fix some standard notation for the rest of the paper. For any measurable set $E \subset \mathbb R$ (or $ E \subset \mathbb T$) we let ${\rm meas} \, (E)$ stand for its Lebesgue measure. The usual max norm in the space of $2\pi$-periodic continuous functions on $[-\pi, \pi]$ will be denoted by $\|\cdot\|_{\infty}.$
Sometimes, to simplify the exposition, the constants will change from line to line, although in several places we will give
the precise values of constants to underline their (in)dependence on parameters.

\section{Auxiliary estimates for the exponential polynomials}\label{aux}

In this section, we first obtain lower estimates for the size of approximations of the constant function $1$ by exponential polynomials.
Then in the next section these estimates will be extended to exponential series by either Baire-category arguments
or inductive constructions.

We start with the following technical lemma.

\begin{lemma}\label{cosAr}
For $\lambda, \gamma\in (0,1],$ define
\[
d_{\lambda,\gamma}=\frac{\sin(\lambda/2)\cos((\gamma+\lambda)/2)}
{\sin(\gamma/2)}.
\]
Then
\begin{equation}\label{bdA}
\frac{\lambda}{4\gamma}\le d_{\lambda,\gamma}\le
\frac{\lambda}{\gamma} \qquad \text{and} \qquad
|(1-e^{i \lambda})+d_{\lambda,\gamma}(1-e^{-i\gamma})|
\le \frac{\lambda(\gamma+\lambda)}{2}.
\end{equation}
\end{lemma}
Since the proof of Lemma \ref{cosAr} is based on simple computations with trigonometrical functions,
 it will be postponed to Appendix A.

The next corollary gives a recipe for constructing polynomials (having, in general, non-integer frequencies) with control
of their size at a fixed point and of their variation on the unit circle.

\begin{cor}\label{numbersP}
Let  $P(\theta)=\sum_{k=1}^n a_k e^{i k\theta}
\in \mathcal A^+.$
For all  $\theta_0\in (0,1/n]$
and $\gamma\in (0,1],$
there exists
$d \in \left[\frac{\theta_0}{4\gamma}, \frac{n \theta_0}{\gamma}\right] $  such that
if
\begin{equation}\label{num4}
P_{d,\gamma}(\theta):=\sum_{k=1}^n \frac{a_k}{1+d} e^{i k \theta}
+ \frac{d}{1+d} e^{i(2\pi-\gamma)\frac{\theta}{\theta_0}},
\end{equation}
then
\begin{equation}\label{num1}
|1- P_{d,\gamma}(\theta_0)|\le 2n\theta_0\gamma
\qquad \text{and} \qquad
\|P_{d,\gamma}'\|_{\infty}\le
n\left(1+\frac{2\pi}{\gamma}\right).
\end{equation}
\end{cor}

\begin{proof}
Let
$\theta_0\in (0,1/n],$ and $\gamma \in (0,1]$ be fixed.
Set
\[
d:=\sum_{k=1}^n a_k d_{k\theta_0,\gamma},
\]
where $d_{k\theta_0,\gamma}, 1 \le k \le n,$ are given by Lemma \ref{cosAr}.
Then, by Lemma \ref{cosAr},
\begin{equation}\label{num3}
\frac{\theta_0}{4\gamma}\le\frac{1}{4\gamma}\sum_{k=1}^n k a_k  \theta_0\le d\le
\frac{1}{\gamma}\sum_{k=1}^n k a_k  \theta_0\le\frac{n \theta_0}{\gamma}.
\end{equation}

Note that
\[
(1+d)(1-P_{d,\gamma}(\theta_0))
=\sum_{k=1}^n a_k [(1-e^{i k \theta_0})+
d_{k\theta_0,\gamma}(1-e^{-i\gamma})].
\]
So using (\ref{bdA}) and \eqref{num3}, we obtain that
\begin{align*}
(1+d)|1-P_{d,\gamma}(\theta_0)|\le&
\sum_{k=1}^n a_k |(1-e^{i k \theta_0})+
d_{k\theta_0,\gamma}(1-e^{-i\gamma})|\\
\le&
\frac{1}{2}\sum_{k=1}^n k a_k\theta_0(\gamma+ k \theta_0)\\
\le&\frac{n \theta_0 \gamma}{2}\sum_{k=1}^n a_k
\left(1+\frac{k \theta_0}{\gamma}\right)\\
\le& 2n\theta_0\gamma (1+d),
\end{align*}
hence the first estimate
in (\ref{num1}) holds.

Finally, by (\ref{num3}),
\[
\|P_{d,\gamma}'\|_{\infty}\le \sum_{k=1}^n k a_k+
d\frac{2\pi-\gamma}{\theta_0}
\le  n+
\frac{2\pi n\theta_0}{\theta_0\gamma}=
n\left(1+\frac{2\pi}{\gamma}\right),
\]
i.e. the second  estimate
in (\ref{num1}) is true.
\end{proof}

Now we are able to show that for any polynomial from $\mathcal A^+$
there is another polynomial close to it in an appropriate weighted norm
and close to the function $1$ on a sequence $(e^{i\theta_m})_{m=1}^{\infty} \subset \T$ going to $1 \in \T.$
\begin{thm}\label{betaT}
Let
$(\epsilon_k)_{k=1}^{\infty}$ be a positive
sequence such that
\[
\underline{\lim}_{k\to\infty}\,\epsilon_k=0 \qquad \text{and} \qquad k \epsilon_k \ge 1, \quad k \in \mathbb N,
\]
and let $\tilde w =(k\epsilon_k)_{k=1}^{\infty}.$
Then for every polynomial $P(\theta)=\sum_{k=1}^n a_k e^{i k\theta}\in  \mathcal{A}^+$
there exist
a  sequence
$(\theta_m)_{m=1}^{\infty} \subset (0,1/n]$ decreasing to zero
and a sequence of polynomials
 $(Q_m)_{m=1}^{\infty} \subset \mathcal{A}^+$
satisfying
\[
\lim_{m\to\infty}\,\frac{|1-Q_m(\theta_m)|}{\theta_m}=0 \qquad \text{and} \qquad
\lim_{m\to\infty}\,\|P-Q_m\|_{\mathcal A(\tilde {w})}=0.
\]
\end{thm}

\begin{proof}
Let a polynomial $P$ be fixed, and let ${\rm deg}\, P=n.$ Define
$e_n:=\sup \{\epsilon_k: 1 \le k \le n \},$
and choose a subsequence
$(\epsilon_{s_m})_{m\ge 1}$ such that
\begin{equation*}\label{Lim11}
\lim_{m\to\infty}\,\epsilon_{s_m}=0.
\end{equation*}
Fix an integer $m> 2\pi n$  such that
\begin{equation}\label{Lim12}
\gamma_m:=\left(\epsilon_{s_m}+\frac{n e_n}{s_m} \right)^{1/2} \le 1, 
\end{equation}
and put
\begin{equation*}
\theta_m:=\frac{2\pi-\gamma_m}{s_m}.
\end{equation*}
Using  Corollary \ref{numbersP}
with $\theta_0=\theta_m$ and $\gamma=\gamma_m$ we conclude that
there exist $d_m$, $0< d_m \le n\theta_m/\gamma_m,$ and a polynomial $Q_m:=P_{d_m,\gamma_m} \in \mathcal A^+,$ ${\rm deg}\, Q_m=s_m>n,$ given by
\[
Q_m(\theta)=\sum_{k=1}^n \frac{a_k}{1+d_m} e^{i k \theta}
+ \frac{d_m}{1+d_m} e^{i s_m \theta},
\]
such that
\[
\frac{|1-Q_m(\theta_m)|}{\theta_m}
\le 2n\gamma_m \qquad \text{and}
\qquad
\| P- Q_m\|_{\mathcal A}=\frac{2d_m}{1+d_m}\le \frac{2n\theta_m}{\gamma_m}.
\]
Hence (\ref{Lim12})
and the latter inequality imply that
\begin{align*}
\| P - Q_m\|_{\mathcal A({\tilde w})} \le \max \left (n e_n, s_m \epsilon_{s_m} \right)
\| P - Q_m \|_{\mathcal A}
\le s_m\gamma_m^2
\frac{2n \theta_m}{\gamma_m}\le  4\pi n \gamma_m.
\end{align*}
Since $\gamma_m \to 0$ as $m \to \infty,$ the statement follows.
\end{proof}

\begin{remark}
Here and in the sequel, the assumption $k \epsilon_k \ge 1, k \in \mathbb N,$ is of a purely technical nature and has been made to simplify our exposition.
\end{remark}

Recall from Section \ref{prelim} that the set
of aperiodic polynomials  
is dense in any $\mathcal{A}^+(w)$.
Thus Theorem  \ref{betaT}
implies the following statement.

\begin{cor}\label{corbeta11}
Let $(\epsilon_k)_{k=1}^{\infty}\subset (0,\infty)$
satisfy
$
\underline{\lim}_{k\to\infty}\,\epsilon_k=0$ and
$k \epsilon_k \ge 1,$ $k \in \mathbb N,$
and let $\tilde w=(k\epsilon_k)_{k=1}^{\infty}.$
Then for every  $f\in \mathcal{A}({\tilde w})$
there exists a sequence of polynomials
 $(Q_m)_{m =1}^{\infty}\in \mathcal A^+$
such that
\begin{equation*}\label{SS11}
\lim_{m\to\infty}\,\inf_{\theta\in(0,1/m]}\,\frac{|1-Q_m(\theta)|}{\theta}=0
\qquad \mbox{and}\qquad
\lim_{m\to\infty}\,\|f-Q_m\|_{\mathcal A(\tilde w)}=0.
\end{equation*}
\end{cor}

The next result is our basic statement, allowing one to spread out the upper estimates for $|1-Q_m|$ proved in Theorem \ref{betaT}
from the sequence $(\theta_m)_{m=1}^{\infty}$  to a larger set containing it.
The result will help us to provide counterexamples on $L^p$-integrability of $(1-z)(1-F)^{-1}.$

\begin{thm}\label{L1}
Let $\psi: (0,1]\mapsto (0,\infty)$
and
$\chi: (0,1]\mapsto (0,\infty)$
be continuous functions  satisfying
\begin{equation}\label{LLL}
\lim_{\theta\to 0+}\,\psi(\theta)=0 \qquad \text{and} \qquad
\lim_{\theta\to 0+}\,\frac{\chi(\theta)}{\psi(\theta)}=0.
\end{equation}
Then for every polynomial $P \in \mathcal A^+$ there exist a sequence
$(\theta_m)_{m=1}^{\infty}\subset (0,1]$  decreasing to zero
and a sequence of polynomials $(Q_m)_{m =1}^{\infty} \subset \mathcal A^+$
such that for all $m \in \N:$
\begin{itemize}
\item [(i)] $m \theta_m \le 2\pi,$
\item [(ii)] $|1-Q_m(\theta_m)|\le 2\psi(\theta_m)\theta_m,$
\item [(iii)] $\| P-Q_m\|_{\mathcal A}\le \frac{2\theta_m}{\chi(\theta_m)}.$
\end{itemize}
Moreover, for each $m \in \mathbb N$ and for each $\theta$ such that $|\theta-\theta_m|\le \psi(\theta_m)\chi(\theta_m)\theta_m$
one has
\begin{equation*}\label{inta}
|1-Q_m(\theta)|\le 10\psi(\theta_m)\theta_m.
\end{equation*}
\end{thm}

\begin{proof}
Let a polynomial $P$ be fixed, and let ${\rm deg}\, P =n.$
From
(\ref{LLL}) it follows that $\lim_{\theta\to 0+}\chi(\theta)=0,$ and then $\lim_{\theta \to 0+}\chi(\theta)\psi(\theta)=0.$
Since we have $\lim_{\theta \to 0+}\chi(\theta)(\psi(\theta))^{-1}=0$ as well, there exists  $\Theta_0\in (0,1]$ such that
\begin{equation}\label{eps}
\chi(\theta)\psi(\theta)+
n \left(\frac{\chi^{1/2}(\theta)}{\psi^{1/2}(\theta)}
+\theta\right)
\le 1,\qquad\theta\in (0,\Theta_0].
\end{equation}
Define
\begin{equation*}\label{FunA}
\tau(\theta):=\frac{2\pi-\gamma(\theta)}
{\theta},\qquad
\gamma(\theta):=\psi^{1/2}(\theta)\chi^{1/2}(\theta),\qquad \theta\in (0,\Theta_0],
\end{equation*}
and note that, in particular,
$n\theta\le 1,$ and $\gamma(\theta)\le 1$ for all $\theta\in (0,\Theta_0].$

Moreover, if $\theta\in (0,\Theta_0]$,
then by (\ref{eps}),
\begin{equation}\label{labD}
n\gamma(\theta)\le
\frac{\psi^{1/2}(\theta)\gamma(\theta)}{\chi^{1/2}(\theta)}
=\psi(\theta),\qquad
\frac{n}{\gamma(\theta)}
\le\frac{\psi^{1/2}(\theta)}{\gamma(\theta)\chi^{1/2}(\theta)}
=\frac{1}{\chi(\theta)}.
\end{equation}

Since  $\gamma$ is continuous on $(0,\Theta_0]$ and
$
\lim_{\theta\to 0+}\,\tau(\theta)=+\infty,
$
there exists a sequence $(\theta_m)_{m\ge m_0}\subset [0,\Theta_0]$ satisfying
\[
\tau(\theta_m)=m,\qquad m\ge m_0.
\]
Moreover, as $\lim_{\theta \to 0+}\gamma(\theta)=0$, we may also assume that
$ m \theta_m\le 2\pi.$

Next, we fix $m \ge \max (m_0,n),$  set $\theta_0=\theta_m$ and
$\gamma=\gamma(\theta_m),$ and apply
 Corollary \ref{numbersP} to the polynomial $P.$   By (\ref{labD}), we infer that
there exist  $d_m >0$ and a polynomial $Q_m=P_{d_m, \gamma(\theta_m)}\in \mathcal A^+$ such  that
\[
|1-Q_m(\theta_m)|
\le 2n\gamma(\theta_m)\theta_m\le 2\psi(\theta_m)\theta_m,
\]
and
\[
\|P - Q_m \|_{\mathcal A}\le \frac{2n\theta_m}{\gamma(\theta_m)}
\le \frac{2\theta_m}{\chi(\theta_m)}.
\]
Moreover, by  \eqref{num1} and (\ref{labD}),
\begin{equation}\label{deriv}
\|Q_m'\|_{\infty} \le  n\left(1+\frac{2\pi}{\gamma(\theta_m)}\right)
\le \frac{\gamma(\theta_m)}{\chi(\theta_m)}
\left(1+\frac{2\pi}{\gamma(\theta_m)}\right)
\le \frac{(1+2\pi)}{\chi(\theta_m)}.
\end{equation}
Using
\begin{equation}\label{IntegrA}
|1-Q_m(\theta)|\le |1-Q_m(\theta_m)|+
\int_{\theta_m}^{\theta}
|Q_m'(s)|\,ds
\end{equation}
 and (\ref{deriv}), we conclude that if
 $|\theta-\theta_m|\le \psi(\theta_m)\chi(\theta_m)\theta_m,$
then
\[
|1-Q_m(\theta)|
\le 2\psi(\theta_m)\theta_m+\frac{1+2\pi}{\chi(\theta_m)}
|\theta-\theta_m|
\le 10 \psi(\theta_m)\theta_m.
\]
\end{proof}
\begin{rem}\label{Rem33}
Let $P\in \mathcal A^+$
and let $Q_m\in \mathcal A^+$  be polynomials given by Theorem \ref{L1}.
If $(w_k)_{k=1}^{\infty}\subset [1,\infty)$ is a nondecreasing  sequence,
then  the estimate from Theorem \ref{L1}, (iii)
 yields
\begin{equation*}\label{Bbeta}
\| P-Q_m\|_{\mathcal A(w)} \le w_m
\| P-Q_m\|_{\mathcal A}\le \frac{2\theta_m w_m}
{\chi(\theta_m)}.
\end{equation*}
In particular, if
$w_k=k^\nu$, $k\in \N,$ for some $\nu\in (0,1)$, then
in view of $\theta _m m\le 2\pi $, $m \in \N$, one has
\begin{equation*}\label{Bbeta0}
\|P-Q_m\|_{\mathcal A(\nu)} \le \frac{2\theta_m m^\nu}
{\chi(\theta_m)}
\le 2(2\pi)^\nu\frac{\theta_m^{1-\nu}}
{\chi(\theta_m)}.
\end{equation*}
\end{rem}

It will be convenient to separate the next easy corollary of Theorem \ref{L1}

\begin{cor}\label{corcos}
Let $\nu \in [0,1)$ and let $\varphi: (0,1]\mapsto (0,\infty)$
be a continuous function such that
\begin{equation}\label{phicond}
\lim_{\theta\to 0+}\,\theta^{1-\nu}\varphi(\theta)=0 \qquad \text{and}
\qquad \lim_{\theta\to 0+}\,\varphi(\theta)=\infty.
\end{equation}
Then for every polynomial $P\in \mathcal A^+$ there exist a sequence
$(\theta_m)_{m=1}^{\infty}\subset (0,1]$ decreasing to zero and a sequence
of polynomials $(Q_m)_{m=1}^{\infty}\subset \mathcal A^+$
such that for all $m \in \N:$
\begin{itemize}
\item [(i)] $\theta_m m \le 2\pi,$
\item [(ii)] $|1-Q_m(\theta_m)|\le 2\varphi(\theta_m)\theta_m^{2-\nu},$ 
\item [(iii)] $\| P-Q_m\|_{\mathcal A(\nu)}\le \frac{2(2\pi)^\nu}{\varphi^{1/2}(\theta_m)}.$
\end{itemize}
Moreover, for each $m \in \N$ and each $\theta$ such that  $|\theta-\theta_m|\le \varphi^{3/2}(\theta_m)
\theta_m^{3-2\nu}$ one has
\[
|1-Q_m(\theta)|\le 10\varphi(\theta_m)\theta_m^{2-\nu}.
\]
\end{cor}

\begin{proof}
Define
\[
\psi(\theta):=\varphi(\theta)\theta^{1-\nu} \qquad \text{and}
\qquad\chi(\theta):=\varphi^{1/2}(\theta)\theta^{1-\nu},\qquad\theta\in (0,1].
\]
Since $\chi$ and $\psi$ satisfy
(\ref{LLL}), the corollary follows from
Theorem \ref{L1} and Remark \ref{Rem33}.
\end{proof}

By density arguments, the next result  follows directly from Corollary  \ref{corcos}.

\begin{cor}\label{cseq}
Let $\varphi: (0,1]\mapsto (0,\infty)$ be
a continuous  function satisfying
(\ref{phicond}), and let $\nu \in [0,1).$
Then for every $f\in \mathcal{A}^+(\nu)$ there exist a sequence of polynomials $(Q_m)_{m=1}^{\infty}\subset \mathcal A^+(\nu)$,
and a sequence $(\theta_m)_{m=1}^{\infty}\subset (0,1]$ decreasing to zero, such that
\[
\lim_{m \to\infty}\|f-Q_m\|_{\mathcal A(\nu)}=0
\]
and
\begin{equation*}\label{Res}
\sup \left \{
\frac{|1-Q_m(\theta)|}{\varphi(\theta_m)\theta_m^{2-\nu}} : |\theta-\theta_m|\le \varphi^{3/2}(\theta_m) \theta_m^{3-2\nu} \right \}\le 10.
\end{equation*}
\end{cor}

\section{Main results}

Using our construction of exponential polynomials from the previous section,
we now produce
a dense set of functions  $F$ ``almost'' satisfying the Erd\H{o}s-Feller-Pollard's  condition \eqref{erdos}
but having such a strong singularity at $1$ that $(1-z)(1-F)^{-1}$ is unbounded in $\D.$
To this aim, we employ either Baire-category arguments (as in Theorem \ref{G}) or, alternatively, an iterative procedure (as in Theorem \ref{G11}).
Thus, we show that the Wiener-type condition \eqref{erdos} is the best one can hope for
as far as the boundedness of $(1-z)(1-F)^{-1}$ is concerned.
We then use power weights $w=(k^\nu)_{k=1}^{\infty}, \nu \in [0,1),$ to construct
examples of $f\in \mathcal A^+(\nu)$ close to the constant function $1$ on a sufficiently large set (but violating \eqref{erdos}). This will be used in the next section
to study the property  $(1-z)(1-F)^{-1} \in L^p(\mathbb T)$ for a fixed $p \in (1,\infty)$ in terms of the Taylor coefficients of $F.$

\begin{prop}\label{G0}
Let  $L_\theta : X \mapsto (0,\infty), \theta \in (0,1],$ be a family of continuous functionals on a complete metric space $(X, \rho),$
and let  $c\ge 0.$
Suppose that
for any $f\in X$ there exists a sequence
$(f_m)_{m = 1}^{\infty}\subset X$ satisfying
\begin{equation}\label{zeroI}
\lim_{m\to\infty}\,\rho (f,f_m)=0 \qquad \text{and} \qquad \limsup_{m\to\infty}\,\inf_{\theta\in (0,1/m]}\,
L_\theta [f_m]\le c.
\end{equation}
Then there exists a residual set $S \subset X$
(in particular, dense in $X$)
such that
\begin{equation}\label{4A0}
\sup_{m\in N}\,\inf_{\theta\in (0,1/m]}\,\,
L_\theta[f]\le c,\qquad f\in S.
\end{equation}
\end{prop}

\begin{proof}
Define a functional $\mathcal F_m$, $m\in \N$,  on $X$ by
$$
\mathcal F_m(f):=\inf_{\theta \in (0,1/m]} L_\theta [f], \qquad f \in X.
$$
Note that
\begin{equation}\label{FfA}
\mathcal F_m(f)\le \mathcal F_n(f),\qquad f\in X,\quad
n\ge m.
\end{equation}
Since $(L_\theta)_{\theta \in (0,1]}$ are continuous, $\mathcal F_m$ is upper-semicontinuous on $X$ for each $m \in \mathbb N,$ by a standard argument.
By e.g. \cite[Theorem 9.17.3]{Giles} for every $m \in \mathbb N$ the set $S_m$ of continuity points of $\mathcal F_m$ is residual.  Hence
\[
S:=\cap_{m\in\N} S_m
\]
is residual as well.
Thus,  by
(\ref{zeroI}) and (\ref{FfA}), for every  $f\in S$ and every $m \in \N:$
\[
\mathcal F_m(f)=\lim_{n\to\infty}\mathcal F_m(f_n)\le
\limsup_{n\to\infty}\,\mathcal F_n(f_n)\le c,
\]
and the statement follows.
\end{proof}

\begin{thm}\label{G}
Let $\nu\in [0,1)$,
and let $\varphi: (0,1]\mapsto (0,\infty)$ be
a continuous  function
satisfying (\ref{phicond}).
Then there exists a residual set $S \subset \mathcal{A}^+(\nu)$  of aperiodic functions with the following property:
for every $f\in S$
there is
a  sequence $(\theta_m)_{m=1}^{\infty}\subset (0,1] $ decreasing to zero such that
\begin{equation}\label{4A}
|1-f(\theta)|\le 10\varphi(\theta_m)\theta_m^{2-\nu},\quad |\theta-\theta_m|\le \varphi^{3/2}(\theta_m)\theta_m^{3-2\nu},\quad m\in \N.
\end{equation}
\end{thm}

\begin{proof}
Consider a family $(L_{\theta_0})_{\theta_0 \in (0,1]}$ of continuous functionals on
$\mathcal{A}^+(\nu)$
given by
\begin{equation}\label{G111}
L_{\theta_0}[f]:=\sup \left \{
\frac{|1-f(\theta)|}{\varphi(\theta_0)\theta_0^{2-\nu}} : |\theta-\theta_0|
\le \varphi^{3/2}(\theta_m)\theta_m^{3-2\nu}\right\},
\end{equation}
for each $\theta_0\in (0,1].$
Using (\ref{Ina2})
and  Corollary \ref{cseq},
we infer that $(L_\theta)_{\theta \in (0,1]}$  satisfies the assumptions  of  Proposition
\ref{G0} with $c=10.$
Therefore, there exists a residual set in $\mathcal{A}^{+}(\nu)$ satisfying (\ref{4A0}) with $c=10.$
Note that the set of aperiodic functions in $\mathcal{A}^{+}(\nu)$ is residual.
Taking the intersection of the two sets, we obtain a residual set satisfying (\ref{4A0}) with $c=10$ again.

\end{proof}

Similarly,  Corollary \ref{corbeta11}
and Proposition \ref {G0} imply the following statement.

\begin{thm}\label{betakN}
Let $(\epsilon_k)_{k=1}^{\infty}$
be a positive sequence such that
\[
\underline{\lim}_{k\to\infty}\,\epsilon_k=0 \qquad \text{and} \qquad k \epsilon_k \ge 1, \quad k \in \mathbb N.
\]
If $\tilde w=(k\epsilon_k)_{k=1}^{\infty}$ and $f\in
\mathcal{A}^{+}(\tilde w),$ then for each $\epsilon>0$  there exists an aperiodic function
$g\in  \mathcal{A}^+({\tilde w})$  such that
\begin{equation}\label{G1beta}
\|f-g\|_{{\mathcal A}({\tilde w})}\le \epsilon \qquad \text{and} \qquad
\lim_{m\to\infty}\,
\frac{|1-g(\theta_m)|}{\theta_m}=0
\end{equation}
for some  sequence
$(\theta_m)_{m=1}^{\infty}$ decreasing to zero.
\end{thm}
Now we present another approach to Theorem
\ref{G} avoiding category arguments. Although the approach leads to a slightly weaker statement,
it seems more transparent.

\begin{thm}\label{G11}
Let $\nu\in [0,1)$  be fixed,
and let $\varphi$ satisfy
(\ref{phicond}).
If $f \in \mathcal A^+(\nu),$ then for all  $\epsilon>0$ and $\delta>0$ there exist an
aperiodic function
$g\in \mathcal{A}^+(\nu)$ and a sequence $(\theta_m)_{m=1}^{\infty} \subset (0,1]$ decreasing to zero such that
$
\| f-g\|_{{\mathcal A}(\nu)}\le \epsilon
$
and
\begin{equation}\label{dd}
|1-g(\theta)|\le (10+\delta)\varphi(\theta_m)\theta_m^{2-\nu} \qquad \text{if} \qquad
|\theta-\theta_m|\le \varphi^{3/2}(\theta_m)\theta_m^{3-2\nu}.
\end{equation}
\end{thm}

\begin{proof}
Let $\epsilon>0$ and $\delta>0$ be fixed.
Without loss of generality we may assume that $f=P_0$ is an aperiodic polynomial and that $\epsilon\in (0,1)$
is so small that the ball
\[
\{h\in \mathcal{A}^{+}(\nu):\,\|f-h\|_{\mathcal{A}(\nu)}\le \epsilon\}
\]
consists of aperiodic functions.

We construct the sequences
\[
(\theta_m)_{m=1}^{\infty}\subset (0,1) \qquad
\text{and} \qquad (\delta_m)_{m=1}^{\infty}\subset  (0,1)
\]
with
\begin{equation}\label{dgamma}
 \sum_{m=1}^\infty \delta_m\le \epsilon,
\end{equation}
and the sequence of polynomials $(P_m)_{m=1}^{\infty}\subset \mathcal A^+(\nu)$
so that, for $1\le m\le N$ and $\theta$ satisfying $|\theta-\theta_m|\le \varphi^{3/2}(\theta_m)\theta_m^{3-2\nu},$
one has
\begin{equation}\label{Ind}
|1-P_N(\theta)|\le
\left(10+\delta\sum_{j=1}^{N-1} \frac{1}{2^j}\right)\,\varphi(\theta_m)\theta_m^{2-\nu},
\end{equation}
and moreover
\begin{equation}\label{Ind1}
\| P_{m-1}-P_m\|_{\mathcal A(\nu)} \le \delta_m.
\end{equation}

Let
\[
\delta_1=\min\{\delta/2,\epsilon/2\}.
\]
By  Corollary \ref{corcos},
there exists (large enough)  $\theta_1\in (0,1)$ and a polynomial
$P_1\in \mathcal A^+(\nu)$
such that
\[
|1-P_1(\theta)|\le
10\varphi(\theta_1)\theta_1^{2-\nu},\]
whenever $
|\theta-\theta_1|\le \varphi^{3/2}(\theta_1)\theta_1^{2-2\nu},$
and
\[
\|P_0-P_1\|_{\mathcal A(\nu)}\le \frac{4\pi}{\varphi^{1/2}(\theta_1)}\le \delta_1.
\]
So, (\ref{Ind}) and (\ref{Ind1}) hold for $N=1$.

Arguing by induction, suppose that (\ref{Ind}) and (\ref{Ind1}) are true for some $N\ge 1$.
Choose $\delta_{N+1}$ satisfying
\[
\delta_{N+1}\le \frac{\delta}{2^{N+1}} \min_{1 \le m \le N} \left(\varphi(\theta_m)\theta_m^{2-\nu}\right)\quad\mbox{and}\quad
\delta_{N+1}\le \frac{\epsilon}{2^{N+1}}.
\]
Then by Corollary  \ref{corcos}
applied to $P=P_N$ there exist
$P_{N+1}\in \mathcal A^+(\nu)$,  and (small enough)
$\theta_{N+1}<\theta_N$ such  that
\begin{equation}\label{Acc}
|1-P_{N+1}(\theta)|\le
10\varphi(\theta_{N+1})\theta_{N+1}^{2-\nu}
\end{equation}
if $|\theta-\theta_{N+1}|\le \varphi^{3/2}(\theta_{N+1})
\theta_{N+1}^{3-2\nu},$
and, moreover,
\[
\| P_N-P_{N+1}\|_{\mathcal A(\nu)}\le
\frac{4\pi}{\varphi^{1/2}(\theta_{N+1})}\le \delta_{N+1}.
\]
Hence for every $m \in [1, N]$ and every $\theta$ satisfying
$|\theta-\theta_m|\le \varphi^{3/2}(\theta_m)\theta_m^{3-2\nu}
$
we have
\begin{align*}
|1-P_{N+1}(\theta)|\le&
|1-P_N(\theta)|+\| P_N-P_{N+1}\|_{\mathcal A(\nu)}\\
\le& \left(10+\delta\sum_{j=1}^{N-1} \frac{1}{2^j}\right)\varphi(\theta_m)\theta_m^{2-\nu}
+ \delta_{N+1}\\
\le& \left(10+\delta\sum_{j=1}^{N-1} \frac{1}{2^j}\right)\varphi(\theta_m)\theta_m^{2-\nu}
+ \delta\frac{\varphi(\theta_m)\theta_m^{2-\nu}}{2^{N+1}}\\
\le&
\left(10+\delta\sum_{j=1}^{N} \frac{1}{2^j}\right)\varphi(\theta_m)\theta_m^{2-\nu}.
\end{align*}
Taking in account (\ref{Acc}), we infer that
 (\ref{Ind}) and (\ref{Ind1}) hold for  $N+1$ too.

By (\ref{dgamma}) the sequence
$(P_m)_{m=0}^{\infty}$ is Cauchy in $\mathcal{A}^+(\nu),$ so there exists $g\in \mathcal{A}^+(\nu)$ such that
\[
\lim_{m\to\infty}\,\| g - P_m\|_{\mathcal A(\nu)}=0 \qquad
\mbox{and}\qquad
\| P_0 - g \|_{\mathcal A(\nu)} \le \sum_{m=1}^\infty \delta_m\le \epsilon.
\]
Finally, (\ref{Ind}) yields (\ref{dd}), and moreover $g$ is aperiodic by the above.
\end{proof}

\section{Regularity of reciprocals in terms of the $L^p$-scale}\label{lp}

In this section we will study the regularity of generating functions for renewal sequences with respect
to the $L^p$-scale.
Namely, we will be concerned with the identifying $p\ge 1$ such that
$
(1-z)(1-F)^{-1} \in L^p (\mathbb T),$
  where $F(z)=\sum_{k=1}^{\infty} a_k z^k, z \in \mathbb D,$
is aperiodic.
It is clear that if 
$$
R_f(\theta):=\frac{\theta}{1-f(\theta)}, \qquad \theta \in [-\pi,\pi],
$$
where $f(\theta)=F(e^{i\theta}),$ then
 $(1-z)(1-F)^{-1} \in L^p (\mathbb T)$ if and only if $R_f \in L^p [-\pi,\pi], p \ge 1.$ Thus,
it is enough to study the same issue for the function $R_f,$ and the results on $(1-z)(1-F)^{-1}$ below will be formulated in terms
of $R_f.$ 
We will show, in particular, that while $R_f$ has a certain amount of regularity, being in $L^p[-\pi,\pi]$ for $p \in [1,2],$
$R_f$ is not very regular in the sense that $R_f$ does not belong, in general, to $L^p[-\pi,\pi],$ if $p >3.$
This result will be put below into a more general (and sharper) context of the spaces $\mathcal A^+(\nu).$

To get positive results on the regularity of $R_f$
we will use an idea from \cite{Littl}.
In particular, we will use the following crucial result proved in \cite[Theorem 1]{Littl}.
(The result formulated in \cite{Littl} has a weaker form but the proof given there yields the statement given below.)
\begin{thm}\label{Lth}
Let
\[
\sum_{k=1}^\infty a_k=1,\qquad a_k\ge 0,
\]
and
\[
f(\theta)=\sum_{k=1}^\infty a_k\cos(\mu_k \theta +\alpha_k),\qquad \theta \in [-\pi,\pi],
\]
with $\mu_k \ge 1$ and $\alpha_k \in \mathbb R, k \ge 1.$
Then
\begin{equation}\label{lth1}
{\rm meas} \, \left(\{\theta \in [-\pi,\pi]: f(\theta)\ge 1-\epsilon \}\right) \le 4\pi\epsilon^{1/2},\qquad \epsilon\in (0,1].
\end{equation}
\end{thm}

Recall that here and in the sequel {\rm meas} stands for the Lebesgue measure.

\begin{rem}\label{litrem}
Note that the above estimate is the best possible as the example of $f(\theta)=\cos \theta$ shows.
Remark also that Theorem \ref{Lth} was stated in \cite{Littl} with a constant $A$ instead of $4\pi$ above.
The uniformity of $A$ was not clarified in \cite{Littl}. Since that property is crucial for our reasoning and to be on a safe side
we provide an independent proof of Theorem \ref{Lth} in Appendix A.
\end{rem}

\begin{cor}\label{corLd}
Let $(a_k)_{k \ge m}\subset [0,\infty), m \in \N,$ and let
$
r=\sum_{k=m}^\infty a_k\in (0,1].
$
If
$$
f(\theta)=\sum_{k=m}^\infty a_k\cos k\theta,
$$
then for every $\theta_0\in (0,1]$ such that $m\theta_0\le 1,$ one has
\begin{equation}\label{lth10}
{\rm meas} \, \left(\{\theta \in [\theta_0,2\theta_0]: r-f(\theta)\le \epsilon \}\right) \le 4\pi\theta_0\sqrt{\frac{\epsilon}{r}},
\qquad \epsilon\in (0,r].
\end{equation}
\end{cor}

\begin{proof}
If
$
\tilde{f}(t)=r^{-1}f(\theta_0(t+1))
$ and $\epsilon\in (0,r]$ then by (\ref{lth1}) we obtain:
\begin{align*}
&{\rm meas}\,
 \left(\{\theta \in [\theta_0,2\theta_0]: r-f(\theta)\le \epsilon \}\right)\\
 =&\theta_0\,{\rm meas}\,\left(
 \{t\in [0,1]: \; \tilde{f}(t)\ge 1-
  \epsilon/r \}
 \right)\\
 \le&
4\pi \theta_0\sqrt{\frac{\epsilon}{r}}.
\end{align*}
\end{proof}

We will also need the next technical estimate for distribution functions.
Its proof is postponed to Appendix A.

\begin{prop}\label{Cake2}
Let
$\Omega\subset [0,\infty)$ be a measurable set of finite measure, and let
$\varphi: \Omega \mapsto (0,\infty)$ be a measurable function.
Suppose that there are constants $r>\eta >0$ and $A>0$ such that
 $\eta\le \varphi\le r$ a.e. on $\Omega$ and, moreover,
   \[
 {\rm meas}\ \left( \{s\in \Omega:\,\varphi(s)\le t \} \right) \le A \sqrt{\frac{t}{r}},\qquad
   t\in [\eta,r].
   \]
Then for all $d\ge 0$ and $p\ge 1,$
\begin{equation}\label{cake4}
\int_{\Omega} \frac{ds}{(\varphi(s)+d)^p}\le
 \frac{{\rm meas}\, (\Omega)}{(r+d)^p}+
\frac{2 A p}{\eta^{p-1/2} \sqrt{r}+d^p}.
\end{equation}
\end{prop}

Now we are ready to prove one of the main results of this section.
It is an extension and sharpening of Littlewood's Theorem $2$ from \cite{Littl}.
For
\begin{equation}\label{repf}
f(\theta)=\sum_{k=1}^\infty a_k e^{i k\theta}, \qquad \theta \in [-\pi,\pi],
\end{equation}
from $\mathcal A^{+},$ define for $\theta\in (0,1]:$
\begin{align*}
r(\theta):=&\sum_{k>1/\theta} a_k,\\
W(\theta):=&\sum_{1\le k\le1/\theta} ka_k,\\
U(\theta):=&\sum_{1\le k\le1/\theta} k^2a_k.
\end{align*}
Note that
\begin{equation}\label{uw}
\theta U(\theta)\le W(\theta).
\end{equation}

\begin{thm}\label{lth2}
Let $f \in \mathcal A^+$ be given by \eqref{repf}
and let $\theta_0 \in (0,1]$.
Then for each $p \ge 1$ there exists $C_p>0$ such that
for every $\theta_0\in (0,1]$ with $r(\theta_0)<1$
one has:
\begin{equation}\label{Cpp}
\int_{\theta_0}^{2\theta_0}
|R_f(\theta)|^p \,d\theta \le
C_p\left\{
\frac{\theta_0}{W^p(\theta_0)}
+\frac{\theta_0^{2-p}r^{p-1}(\theta_0)}
{W^{2p-1}(\theta_0)+r^{p-1}(\theta_0)\theta_0 U^p(\theta_0)}\right\}.
\end{equation}
\end{thm}

\begin{proof}
Let $n=n(\theta_0)$
satisfy
$
n\le 1/\theta_0<n+1,
$ and let $p \ge 1$ be fixed.
By assumption,
\begin{equation}\label{Assumtion}
W(\theta_0)\ge \sum_{k=1}^{n} a_k=1-r(\theta_0)> 0 \qquad \text{and}\qquad
U(\theta_0)\ge \sum_{k=1}^{n} a_k > 0.
\end{equation}

If  $\theta \in [\theta_0,2\theta_0]$  and $1 \le k \le n$ then
$0<k\theta\le 2n\theta_0 <3\pi/4$ so that
 $\sin k\theta\ge k\theta/4$ and
 $\sin k\theta/2\ge k\theta/\pi$.
Hence for $\theta\in [\theta_0,2\theta_0]$ we have
\begin{equation}\label{Ll3}
\sum_{k=1}^n a_k\sin k\theta
\ge \frac{\theta_0}{4}\sum_{k=1}^n k a_k \ge \frac{\theta_0}{4}W(\theta_0),
\end{equation}
and
\begin{align}\label{AAr}
\sum_{k=1}^n a_k(1-\cos k\theta)=&
2\sum_{k=1}^n a_k \sin^2(k\theta/2)\\
\ge& \frac{2\theta^2}{\pi^2}\sum_{k=1}^n k^2 a_k\ge
 \frac{2\theta_0^2}{\pi^2}U(\theta_0).\notag
\end{align}

If $r(\theta_0)=0$, then by
(\ref{Ll3}) we have
\[
\int_{\theta_0}^{2\theta_0}|R_f(\theta)|^p \,d\theta \le
\int_{\theta_0}^{2\theta_0}
\frac{\theta^p\,d\theta}
{(\theta_0/4W(\theta_0))^p}\le 8^p\frac{\theta_0}{W^p(\theta_0)},
\]
and (\ref{Cpp}) holds.

Let now $r(\theta_0)>0$. For $y\in [0,r(\theta_0)]$ let $S(y)$  be the set
of such $\theta\in [\theta_0,2\theta_0]$ that
\begin{equation}\label{Ll4}
\varphi(\theta):=\sum_{k=n+1}^\infty a_k(1-\cos k\theta)
\le y.
\end{equation}
Corollary \ref{corLd} then yields
\begin{equation}\label{Ll5}
|S(y)|\le 4\pi\theta_0\frac{\sqrt{y}}{\sqrt{r(\theta_0)}}, \qquad
y \in [0,r(\theta_0)].
\end{equation}

Now, since
\[
\sin^2 k\theta=1-\cos^2 k\theta\le2(1-\cos k\theta),
\]
 by (\ref{Ll4}) again, we have
\begin{align*}
\left|\sum_{k=n+1}^\infty a_k\sin k\theta\right|\le&
\left(\sum_{k=n+1}^\infty a_k\right)^{1/2}
\left(\sum_{k=n+1}^\infty a_k\sin^2 k\theta\right)^{1/2}
 \notag \\
\le& \sqrt{2} r^{1/2}(\theta_0) \varphi^{1/2}(\theta)\\
\le&
\sqrt{2 r(\theta_0)y},\qquad \theta\in S(y). \notag
\end{align*}
So, putting
$$
\eta:=\min \left\{r(\theta_0), \frac{\theta_0^2 W^2(\theta_0)}{128 r(\theta_0)} \right\},
$$
we infer that for every $\theta \in S(\eta):$
\begin{align}\label{im}
|{\rm Im}\,f(\theta)|\ge&
\left|\sum_{k=1}^na_k\sin k\theta\right|-\left|\sum_{k=n+1}^\infty a_k\sin k\theta\right|\\
\ge& \frac{\theta_0 W(\theta_0)}{4}-
\sqrt{2 r(\theta_0)\eta} \notag  \notag \\
\ge& \frac{\theta_0 W(\theta_0)}{8}.\notag
\end{align}

We now estimate the left-hand side of \eqref{Cpp} as follows. Write
\[
\int_{\theta_0}^{2\theta_0}
|R_f(\theta)|^p \,d\theta \le J_{p,1}(\theta_0)+J_{p,2}(\theta_0),
\]
where
\begin{align*}
J_{p,1}(\theta_0)&:=2^p\theta_0^p
\int_{\theta\in S(\eta)}\frac{d\theta}
{|{\rm Im}\,f(\theta)|^p},\\
J_{p,2}(\theta_0)&:=2^p\theta_0^p
\int_{\theta\in [\theta_0,2\theta_0]\setminus S(\eta)}\frac{d\theta}
{|1-{\rm Re}\,f(\theta)|^p}.
\end{align*}

We deal with each of the terms $J_{p,1}$ and $J_{p,2}$ above separately.
First, observe that by \eqref{im},
\begin{equation}\label{Jp1}
J_{p,1}(\theta_0)
\le 2^{4p}\frac{\theta_0}{W^p(\theta_0)}.
\end{equation}

Second, if in addition $\eta \le r(\theta_0),$ then
\[
\frac{\theta_0}{r(\theta_0)}\le \frac{8\sqrt{2}}{W(\theta_0)}.
\]
Setting to simplify the notation
\[
r=r(\theta_0),\quad W=W(\theta_0), \quad U=U(\theta_0),
\quad d=\frac{2\theta_0^2 U(\theta_0)}{\pi^2},
\]
and using Proposition \ref{Cake2}
and (\ref{Ll5}), we obtain that
\begin{align*}
\frac{J_{p,2}(\theta_0)}{2^p}
\le &
\theta_0^p \int_{\theta\in [\theta_0,2\theta_0]\setminus S(\eta)}\frac{d\theta}
{(\varphi(\theta)+d)^p}
\\
\le&
\frac{\theta_0^{p+1}}{(r+d)^p}+
\frac{8\pi p  \theta_0^{p+1}}{\eta^{p-1/2}\sqrt{r}+d^p}\\
\le&
\frac{\theta_0^{p+1}}{r^p}+
\frac{8\pi p \theta_0^{p+1}}{(\theta_0 W)^{2p-1}\sqrt{r}/(128 r)^{p-1/2}+d^p}\\
\le& \frac{(8\sqrt{2})^p \theta_0}{W^p}\\
+&\frac{8\pi p\theta_0^{p+1}}
{\theta_0^{2p-1} W^{2p-1}\sqrt{r}/
(128 r)^{p-1/2}+
\theta_0^{2p}U^p/(2/\pi^2)^{p}}\\
\le&
 \frac{(8\sqrt{2})^p \theta_0}{W^p}
+\frac{8\pi p  (128)^{p-1/2}}{\theta_0^{p-2}( W^{2p-1}/ r^{p-1}+
\theta_0U^p)}.
\end{align*}
Taking in account (\ref{Jp1}) we infer that (\ref{Cpp}) holds with
$
C_p=3p\cdot 2^{8p}.
$
\end{proof}

Theorem \ref{lth2} allows us to describe the integrability of $1/(1-f)$ in terms of the size of Fourier coefficients of $f \in \mathcal
A^+$.
We will need the next simple proposition on series with positive terms, proved in Appendix A.
\begin{prop}\label{n2}
Let $(q_k)_{k\ge 2^{m-1}+1}, m \in \N,$ be a
 positive decreasing sequence, and let $\alpha\ge 0$.
Then
\[
\frac{1}{2^\alpha}\sum_{n=2^{m}}^\infty n^{\alpha -1} q_n\le
\sum_{n=m}^\infty 2^{\alpha n} q_{2^n}\le
2^{1+\alpha}\sum_{n=2^{m-1}+1}^\infty n^{\alpha -1} q_n.
\]
\end{prop}

For $f\in \mathcal{A}$ define
\[
r_n:=r(1/n),\quad
W_n:=W(1/n),\quad
U_n:=U(1/n),\quad n\in \N.
\]

\begin{cor}\label{genL}
Let $f \in \mathcal A^+$ be given by \eqref{repf}.
Then there exist  $c >0$ and $m_0\in \N$ such that
\begin{equation}\label{L1C1}
c^{-1}\sum_{n=m_0}^\infty\frac{1}{nW_n+n^2 r_n}\le
\int_0^{1/m_0}
\frac{d\theta}{|1-f(\theta)|} \le c\sum_{n=m_0}^\infty\frac{1}{nW_n}.
\end{equation}
In particular, if $f$ is aperiodic and the right-hand side of \eqref{L1C1} is finite, then
$
1/(1-f) \in L^1[-\pi,\pi]$.
\end{cor}

\begin{proof}
Choose $\Theta_0 \in (0,1]$ with $r(\Theta_0)< 1$.
By  Theorem \ref{lth2},
\begin{equation}\label{CppL1}
\int_{\theta_0}^{2\theta_0}
\frac{d\theta}{|1-f(\theta)|} \le
\theta_0^{-1}\int_{\theta_0}^{2\theta_0}
|R_f(\theta)| \,d\theta
\le \frac{c}{W(\theta_0)},\quad \theta_0\in (0,\Theta_0],
\end{equation}
for a constant $c>0.$
Fix
  $n_0\in \N$ such that
  \[
   m_0:=2^{n_0-1}\ge 1/\Theta_0,
   \]
   and note that $W_n > 0, n \ge m_0,$ and moreover $(W_n)_{n=m_0}^{\infty}$  monotonically increases.
Using (\ref{CppL1}) and Proposition \ref{n2},
we obtain:
\[
\int_0^{1/m_0}
\frac{d\theta}{|1-f(\theta)|}
=\sum_{n=n_0}^\infty
\int_{1/2^{n}}^{1/2^{n-1}}
\frac{d\theta}{|1-f(\theta)|}
\]
\[
\le c\sum_{n=n_0}^\infty
\frac{1}{W_{2^n}}\le 2c\sum_{n=2^{n_0-1}+1}^\infty
\frac{1}{nW_n}
 \le 2c\sum_{n=m_0}^\infty
\frac{1}{nW_n},
\]
that is, the right-hand side estimate in \eqref{L1C1} holds.
If the series $\sum_{n=m_0}^\infty{1}/(nW_n)$ converges, then since $f$ is aperiodic and $|R_f(\theta)|$ is symmetric in the sense that $|R_f(\theta)|=|R_f(-\theta)|,$
we infer that $
1/(1-f) \in L^1[-\pi,\pi].$

To prove the left-hand side estimate in \eqref{L1C1}, we note that if $n\in \N$ and $\theta \in (0,1]$ are such that
$n\le 1/\theta <n+1$,  then, using
(\ref{uw}), we have
\begin{align}\label{Ncond}
|1-f(\theta)|\le& \sum_{k=1}^\infty a_k(1-\cos k\theta)+
\left|\sum_{k=1}^\infty a_k\sin k\theta \right|
\\
\le& \frac{\theta_0^2}{2}U_n+\theta_0 W_n+
2r_n \notag\\
\le&  2( W_n/n+r_n).\notag
\end{align}
Therefore, from (\ref{Ncond})
it follows that there exists $m_0\in\N$ such that
\begin{align*}
\int_0^{1/m_0}\frac{d\theta}{|1-f(\theta)|}
=&\sum_{n=m_0}^\infty \int_{1/(n+1)}^{1/n}\frac{d\theta}{|1-f(\theta)|}
\ge \frac{1}{2}\sum_{n=m_0}^\infty
\frac{1}{W_n/n+r_n}
\int_{1/(n+1)}^{1/n}d\theta\\
=&
\frac{1}{2}\sum_{n=m_0}^\infty
\frac{1}{(W_n/n+r_n)n(n+1)}
\ge \frac{1}{4}\sum_{n=m_0}^\infty
\frac{1}{nW_n+n^2r_n}.
\end{align*}
(Since we do not need Theorem \ref{lth2} for the left-hand side estimate,
 a dyadic partition of $[0,1/m_0]$ is replaced with a partition, in a sense, more convenient
 for writing down the final estimate.)
\end{proof}

\begin{rem}\label{RemL1}
From (\ref{CppL1}) it follows that if
$f\in \mathcal{A}^{+}$ is given by \eqref{repf} and 
\[
\mu=\sum_{k=1}^\infty k a_k=\infty,
\]
then
\[
\lim_{\theta_0\to 0+}\,\int_{\theta_0}^{2\theta_0}
\frac{d\theta}{|1-f(\theta)|}=0,
\]
cf. Littlewood's result \eqref{Lpart}.

If $\mu<\infty,$
then
\begin{equation}\label{deltaL}
\int_0^\delta
\frac{d\theta}{|1-f(\theta)|}=\infty
\end{equation}
for every $\delta>0,$
since in this case
\[
W_n + n r_n=\sum_{k=1}^n k a_k + n\sum_{k=n+1}^\infty a_k
\le \mu, \qquad n \in \mathbb N,
\]
and
\[
\sum_{n=1}^\infty
\frac{1}{nW_n+n^2 r_n}\ge
\frac{1}{\mu}\sum_{n=1}^\infty \frac{1}{n}=\infty.
\]
On the other hand,  in this case (\ref{deltaL}) is a direct consequence of
$\lim_{\theta\to 0+}\,(f(\theta)-1)/\theta=\mu.$
\end{rem}
\begin{remark}
For aperiodic $f(\theta)=\sum_{k=1}^\infty a_k e^{ik\theta} \in  \mathcal A^+$,  let $F(z):=\sum_{k=1}^\infty a_k z^k, z \in \overline{\mathbb D}.$
Recall that $1/(1-F)$ is analytic in $\mathbb D$ and continuous in ${\mathbb D}\setminus\{1\}.$ 
If $1/(1-f)\in L^1[-\pi,\pi],$
then $1/(1-F) \in L^1(\mathbb T),$ hence 
$1/(1-F)$ belongs to the Hardy space $H^1(\mathbb D).$
If $1/(1-F(z))=\sum_{k=0}^{\infty} b_k z^k, z \in \mathbb D,$
then by Hardy's inequality (\cite[IX.9.7]{Ineq}),
\[
\sum_{n=1}^\infty n^{-1}b_n<\infty.
\]
\end{remark}

We pause now to illustrate Corollary \ref{genL} by the following example.

\begin{example}\label{eXL1}
Consider
\[
f_\epsilon(\theta):=c_\epsilon\sum_{k=1}^\infty
\frac{\log^\epsilon(k+1)}{k^2}e^{ik\theta},\qquad
f_\epsilon(0)=1,\quad \epsilon\ge 0.
\]
Then $f_\epsilon \in \mathcal{A}^{+},$ and it is aperiodic. For each $\epsilon>0$ we have
\[
W_n=\sum_{k=1}^n \frac{\log^\epsilon(k+1)}{k}\ge
c \log^{1+\epsilon} (n+1),
\]
hence   $1/(1-f_\epsilon) \in L^1[-\pi,\pi]$ by \eqref{L1C1}.
On the other hand, if $\epsilon=0$
then
\[
nW_n+n^2r_n
\le n\sum_{k=1}^n \frac{1}{k}+
n^2\sum_{k=n+1}^\infty \frac{1}{k^2}\le c n \log (n+1),
\]
for some constant $c>0,$ so  \eqref{L1C1} implies that  $1/(1-f_0) \not\in L^1[-\pi,\pi]$.
(Similarly,
if $f \in \mathcal{A}^{+}$ is given by
\[
f(\theta):=\sum_{k=1}^\infty \frac{e^{ik\theta}}{k(k+1)}=
1+(1-e^{-i\theta})\log (1-e^{i\theta}),
\]
then $f$ is aperiodic and
$1/(1-f) \not \in L^1[-\pi,\pi]$.)
\end{example}

Although, in general, $R_f \not \in L^p[-\pi,\pi], p >3,$ for $f \in {\mathcal A}^+,$ it is possible to formulate a sufficient
condition on the Fourier coefficients $f$ ensuring $R_f \in L^p[-\pi,\pi]$ for fixed $p > 2.$
(This way we may also produce $f \in {\mathcal A}^+$ such that $R_f \in \cap_{p >2} L^p[-\pi,\pi].$)
The next statement is a direct implication of Theorem \ref{lth2}.

\begin{cor}\label{CL2p10}
Let $f\in \mathcal{A}^+$ be aperiodic, and let $p >2.$
If there is $m_0\in \N$ such that
\begin{equation}\label{Mclll}
\sum_{n=m_0}^\infty \frac{n^{p-2}r_n^{p-1}}{
nW_n^{2p-1}+r_n^{p-1}U_n^p}<\infty,
\end{equation}
then $R_f \in L^p[-\pi,\pi]$.
\end{cor}

\begin{proof}
By aperiodicity $f$ and symmetry of $|R_f(\theta)|$
it suffices to prove that
\begin{equation}\label{pint}
\int_0^\delta |R_f(\theta)|^p\,d\theta<\infty
\end{equation}
for some $\delta>0$.
Using (\ref{Mclll}), Theorem \ref{lth2}
and Proposition \ref{n2},
we infer that there exists $c>0$ such that for large enough $n_0$
and $m_0=2^{n_0-1}+1$:
\begin{align*}
\int_0^{1/2^{n_0-1}}
|R_f(\theta)|^p\,d\theta
&=\sum_{n=n_0}^\infty
\int_{1/2^n}^{1/2^{n-1}}
|R_f(\theta)|^p\,d\theta\\
&\le c\sum_{n=1/2^{n_0}}^\infty
\left[\frac{1}{2^{n}W_{2^n}^p}
+\frac{2^{(p-1)n}r_{2^n}^{p-1}}
{2^nW_{2^n}^{2p-1}+r_{2^n}^{p-1}U_{2^n}^p}\right]\\
&\le c 2^p\sum_{n=m_0}^\infty
\left[\frac{1}{n^2W_n^p}
+\frac{n^{p-2}r_n^{p-1}}{nW_n^{2p-1}+r_n^{p-1}U_n^p}\right]<\infty.
\end{align*}
\end{proof}

Further we will make use of Theorem \ref{lth2} to describe the regularity $R_f$ with respect to the $L^p$-scale
for arbitrary  $f\in \mathcal A^+(\nu), \nu >0.$

\begin{cor}\label{CL2p}
Let $\nu\in (0,1),$ and $f\in \mathcal{A}^+(\nu)$  be aperiodic.
Then
 $R_f \in  L^p[-\pi,\pi],$ where $p=1+\frac{1}{1-\nu}.$
On the other hand, for any $\nu\in [0,1)$ and
\[
p\in \left(2+\frac{1}{1-\nu},\infty\right)
\]
there exists an aperiodic function $f\in \mathcal{A}^+(\nu)$ such that
$R_f \not\in L^p[-\pi,\pi].$
\end{cor}

\begin{proof}
Let us prove the first claim. Suppose $f(\theta)=\sum_{k=1}^\infty a_k e^{ik\theta}, \theta \in [-\pi,\pi].$ Again, by aperiodicity of $f$ and symmetry of $|R_f(\theta)|$
it suffices to show that that \eqref{pint} holds for some $\delta>0$
and $p=1+1/(1-\nu).$
Choose $n_0\in \N$ so large that
$W_n\ge 1/2$, $n\ge n_0.$ Then,  taking into account \cite[Thm. 165]{Ineq}, we conclude that
there are constants $c>0$ and $\tilde{c}>0$ such that
\begin{align*}
\sum_{n=n_0}^\infty \frac{n^{p-2}r_n^{p-1}}{
nW_n^{2p-1}+r_n^{p-1}U_n^p}\le&
2^{2p-1}\sum_{n=n_0}^\infty n^{p-3}
\left(\sum_{k=n+1}^\infty a_k\right)^{p-1}\\
\le& c
\sum_{
n=1}^\infty a_n
\left(\sum_{k=1}^n k^{p-3}\right)^{1/(p-1)}\\
\le& \tilde{c}
\sum_{n=1}^\infty a_n
n^{(p-2)/(p-1)}\\
=&\tilde{c}
\sum_{n=1}^\infty a_n
n^\nu<\infty.
\end{align*}
Thus, Corollary \ref{CL2p10} implies that
 $R_f \in L^p[-\pi,\pi].$

Second, to prove the negative result,
let $p \in \left(2+\frac{1}{1-\nu},\infty\right)$ be fixed.
Write
\begin{equation*}\label{ppp}
p=2+\frac{1}{1-\nu}+\delta,
\end{equation*}
and for a fixed $\epsilon\in (0,1-\nu)$
define a continuous function
\[
\varphi(\theta):=\theta^{-\epsilon},\qquad \theta\in (0,1].
\]
Since $\varphi$ satisfies (\ref{phicond}),
Theorem \ref{G} implies that
there exists an aperiodic  $f=f_\epsilon\in\mathcal{A}^+(\nu)$  and a decreasing sequence $(\theta_m)_{m=1}^{\infty}\subset (0,1]$ tending to zero such that
\[
|1-f(\theta)|\le c \theta_m^{2-\nu-\epsilon}
\]
whenever $|\theta-\theta_m|\le \theta_m^{3-2\nu}.$
Setting $$\beta=(1-\nu-\epsilon)p-3+2\nu$$ we have
\begin{align}\label{blow}
\int_{\theta_m}^{\theta_m+\theta_m^{3-2\nu}}
\frac{\theta^p\,d\theta}{|1-f(\theta)|^p}
\ge \frac{c}{\theta_m^{p(2-\nu-\epsilon)}}
\int_{\theta_m}^{\theta_m+\theta_m^{3-2\nu}}
\,\theta^p\,d\theta
\ge
\frac{c}{\theta_m^\beta}
\end{align}
for some constant $c>0.$
Note that
\[
\beta=(1-\nu)\delta-\epsilon(3-2\nu+\delta),
\]
and choose $\epsilon$ such that $\beta>0,$ that is
\[
0<\epsilon <\frac{(1-\nu)\delta}{3-2\nu+\delta}.
\]
As $\theta_m\to 0$, $m\to\infty$, the right-hand side of \eqref{blow} tends to infinity
as $m\to\infty,$
hence $R_f \not \in L^p[-\pi,\pi].$
\end{proof}

Finally, as a  consequence of Corollary \ref{CL2p}, we derive a result on the regularity of $R_f$ measured
in terms of $L^p$-spaces. The result  corresponds formally to the case $\nu=0$ in Corollary \ref{CL2p} and should be compared
to the property \eqref{BEQ} discussed by Erd\H{o}s, de Bruijn and Kingman. While its positive part is elementary, it was apparently overlooked
by specialists in probability theory.

\begin{cor}
Let $f(\theta)=\sum_{k=1}^\infty a_k e^{ik\theta}\in \mathcal A^+$ be aperiodic, and let $(b_k)_{k=0}^{\infty}$ be a renewal sequence associated to $(a_k)_{k=1}^{\infty}.$
Then $R_f \in L^2[-\pi,\pi]$ and
\begin{equation}\label{pars}
\sum_{k=0}^{\infty} (b_{k+1}-b_k)^2<\infty.
\end{equation}
At the same time, there exists an aperiodic $f \in \mathcal A^+$ such that
 $R_f \not \in L^p[-\pi,\pi]$ for every $p \in (3,\infty].$
\end{cor}
\begin{proof}
The second claim follows directly from Corollary \ref{CL2p}.
Let $F(z):=\sum_{k=1}^\infty a_k z^k, z \in \overline{\mathbb D},$ 
so that $f(\theta)=F(e^{i\theta}), \theta \in [-\pi,\pi].$ Then
 $1/(1-F(z))=\sum_{k=0}^{\infty} b_k z^k, z\in \D.$ Note that
${\rm Re}\, (1/(1-F))\in L^1(\mathbb T)$ since
$$
\int_{-\pi}^{\pi} {\rm Re}\, \frac{1}{1-F(e^{i\theta})} \, d\theta \le \limsup_{r \to 1} \int_{-\pi}^{\pi} {\rm Re}\, \frac{1}{1-F(re^{i\theta})}\, d\theta=2\pi
$$
by Fatou's Lemma and positivity of the harmonic function ${\rm Re}\, (1/(1-F))$
in $\D$,
cf. \cite[p. 10-12]{Kingman}.
On the other hand, if $n_0$ is such that $a_{n_0}\not = 0$ and  $\theta \in [-1/n_0,1/n_0],$ then
$$
{\rm Re} \, (1-F(e^{i\theta}))\ge a_{n_0}(1-\cos n_0 \theta)\ge \frac{2\theta^2}{\pi^2}n_0^2 a_{n_0}
$$
and
$$
 {\rm Re}\, \frac{1}{1-F(e^{i\theta})}=\frac{{\rm Re} \, (1-F(e^{i\theta}))}{|1-F(e^{i\theta})|^2} \ge  2a_{n_0}\left(\frac {n_0}{\pi}\right)^2
 (R_f(\theta))^2.
$$
Thus  $R_f \in L^2[-1/n_0, 1/n_0],$
and, if $f$ is aperiodic, then $R_f\in L^2[-\pi,\pi].$
As the latter property is equivalent to $(1-z)(1-F)^{-1} \in L^2(\mathbb T)$, Parseval's identity yields \eqref{pars}.
\end{proof}

\begin{rem}\label{measure}
For $f \in \mathcal A^+$ and $G_\epsilon:=\{\theta \in [-\pi,\pi]:  {\rm Re} \, f (\theta) \ge 1-\epsilon \}$ there are several estimates in the literature of the form
$$
{\rm meas}\, (G_\epsilon)
\le C \sqrt \epsilon
$$
with an absolute constant $C>0,$ see e.g. \cite{Ball}, \cite{Des}, \cite{Des2},  \cite{Yudin}, \cite{YudinP}. The estimates are
motivated by applications in probability theory and number theory, and they seem to be weaker than the estimate \eqref{lth1}
provided by Theorem \ref{Lth}. (A related bound for  $f$ in terms of its coefficients has been given in \cite{Bened}.)
To clarify their relations to our treatment,
for $\epsilon \in (0,1]$ and $f\in \mathcal{A}^+$ define
$$ E_\epsilon:=\{\theta\in [-\pi,\pi]:
|1-  \, f(\theta)| \le \epsilon \}.$$
Note that
$$
{\rm meas}\, (E_\epsilon) \le {\rm meas}\, (G_\epsilon) \le C \sqrt \epsilon.
$$
On the other hand, if $f$
is  aperiodic then Littlewood's theorem \eqref{aper} implies that
 for every $\alpha\in (0,1)$ there exists  $c_\alpha(f)>0$ such that
\begin{equation}\label{epsilon}
{\rm meas} \,(E_\epsilon)\le c_\alpha (f) \epsilon^\alpha.
\end{equation}
Indeed, let $\alpha\in (0,1)$ be fixed.
Then, taking
 $\gamma\in (0,1-\alpha)$ and using \eqref{aper}, we have
\begin{align*}
\epsilon^{-\alpha}{\rm meas }\,(E_\epsilon) \le&  \int_{E_\epsilon} \frac{d\theta}{|1-f(\theta)|^\alpha}\le
\int_{-\pi}^\pi \frac{d\theta}{|1-f(\theta)|^\alpha}\\
\le& \left(\int_{-\pi}^\pi \frac{\theta^{\gamma/\alpha}\,d\theta}
{|1-f(\theta)|}\right)^{\alpha}
\left(\int_0^\pi \frac{d\theta}{\theta^{\gamma/(1-\alpha)}}
\right)^{1-\alpha}=:c_\alpha(f),
\end{align*}
and (\ref{epsilon}) follows.
Thus \eqref{aper}  gives asymptotically better bounds for ${\rm meas}\, (E_\epsilon).$ However, it is not clear
whether the constant $c_\alpha(f)$
can be taken independent of $f.$
\end{rem}

\section{Remarks on $p$-functions}

The theory of $p$-functions can be considered as a continuous counterpart of the theory of renewal sequences.
Its basic facts can be found in \cite{Kingman}.
To put the relevant considerations on $p$-functions into our setting, let us first recall a couple of basic facts.

Let a function $\varphi$ in the right half-plane $\mathbb C_+:=\{z\in \mathbb C: {\rm Re}\, z >0\}$  be defined as
\begin{equation}\label{Form}
\varphi(z)=z+c+\int_{(0,\infty)} (1-e^{-zt})\,\nu(dt),
\end{equation}
 where $c\ge 0,$ and $\nu$ is a positive Borel measure on $(0,\infty)$ satisfying
\[
\int_{(0,\infty)}
\frac{t\,\nu(dt)}{1+t}<\infty.
\]
(It is easy to see that $\varphi$ is analytic in $\mathbb C_+$ and continuous in $\overline{\mathbb C}_+.$)
Then there is a unique continuous function $g: [0,\infty)\mapsto [0,1], g(0)=1,$ called a standard $p$-function, such that
the Laplace transform $\mathcal L g$ of $g$ can be represented as
\begin{equation}\label{Pf}
(\mathcal L g)(z):= \int_0^\infty e^{-zt} g(t)\,dt=\frac{1}{\varphi(z)},
\end{equation}
for $z >0.$ Observe that if $g$ is a standard $p$-function, then $\mathcal L g$ extends analytically to $\mathbb C_+$ and continuously to
$\overline{\mathbb C}_+\setminus\{0\},$ see \cite[p. 74]{Kingman} and \cite[Theorem 5]{Kingman2}.
For $p$-function $g$ given by  (\ref{Pf}) and (\ref{Form})
we will write
$g\sim (c,\nu)$.

We refer to \cite[Chapter 3]{Kingman} concerning basic facts of the analytic theory of $p$-functions.
Note that $\varphi$ in \eqref{Form} is a so-called Bernstein function, and the class of $p$-functions is an important subclass of a class of potential measures arising in the study of Bernstein functions. For a thorough discussion of Bernstein functions and associated potential measures, see \cite[Ch. 5, p. 63-64 and Ch. 11]{Schilling}.

Moreover, by \cite[Theorem 6]{Kingman2}, for $g \sim (c,\nu)$ one has
\[
g(\infty):=\lim_{t\to\infty}\,g(t)=\begin{cases}
\frac{1}{1+\int_{(0,\infty)}\,t\,\nu(dt)}\in [0,1],& \qquad
\mbox{if}\quad c=0,\\
0,& \qquad \mbox{if}\quad c > 0.
\end{cases}
\]
It was proved in  \cite[Theorem 3]{Kingman1} (see also \cite[p. 75-76]{Kingman})  that if
$g(\infty)>0$, i.e. if
\begin{equation}\label{Kin}
c=0\quad \mbox{and}\quad \int_{(0,\infty)} t\nu(dt)<\infty,
\end{equation}
 then $g$ has bounded variation on $[0,\infty).$
 As in the setting of renewal sequences, a natural question is whether $g$ has always bounded variation, i.e. also in the case when $g(\infty)=0.$
 The question was asked by J. Kingman in \cite[p. 76]{Kingman}, and soon after J. Hawkes  produced in \cite{Hawkes} an example showing that the answer is ``no'' in general.

The argument in  \cite{Hawkes} was based on the following
observation. Let
\[
\varphi_0(z):=\int_0^\infty (1-e^{-zt})\,\nu(dt), \qquad \int_{(0,\infty)}\,\nu(dt)<\infty, \quad
{\rm Re}\,z\ge 0,
\]
and the corresponding $p$-function $g$ be given by
\[
\int_0^\infty e^{-zt}g(t)dt=\frac{1}{z+\varphi_0(z)}.
\]
If $g$ has bounded variation on $[0,\infty)$ and $g(\infty)=0,$
then
\begin{equation}\label{lmC}
\lim_{\theta\to 0+}\,
\frac{i\theta}{i\theta-\varphi_0(-i\theta)}=0.
\end{equation}
Essentially, Hawkes constructed
a quasi-exponential series $$f(\theta)=\sum_{k=1}^{\infty}a_k e^{i\lambda_k\theta}, \qquad a_k \ge 0,$$
(see the introduction)
such that
\[
\sum_{k=1}^{\infty} a_k < \infty, \qquad \sum_{k=1}^{\infty} k a_k = \infty \qquad \text{and} \qquad \underline{\lim}_{\theta\to 0+}\,
\frac{|1-f(\theta)|}{\theta}=0
\]
(thus ${\lim}_{\theta\to 0+}\,
{|1-f(\theta)|}/{\theta}$ does not exist  
).
Then, setting
\[
\varphi_0(z):=1-\sum_{k=1}^\infty a_k e^{-\lambda_k z}
\]
for $z$ with  ${\rm Re}\, z \ge 0,$
one obtains the desired (counter-)example.

One can prove that in Hawkes's example
\begin{equation}\label{13}
\sum_{k=1}^\infty k^{\alpha} a_k<\infty,\qquad\alpha\in [0,1/3).
\end{equation}
In other words, the example states that
there exists a finite (discrete) Borel measure $\nu$ on $(0,\infty)$,
satisfying
\[
\int_{(0,\infty)}\,t^\alpha \nu(dt)<\infty,\qquad \alpha\in [0,1/3),
\]
such that the corresponding $p$-function
$g \sim (0,\nu)$ has unbounded variation on $[0,\infty).$
If one can arrange ${\rm supp} \, \nu \subset \N,$ then the example could also be used to produce a negative answer
to the question by Erd\H{o}s-de Bruijn-Kingman on renewal sequences.
However,  we do not see how to realize that in a way different from the above.

On the other hand, using our results, we can generalize the considerations by
Hawkes  in the following way, thus showing that the condition \eqref{Kin} is best
possible in a sense (as in \eqref{erdos}, the discrete analogue of \eqref{Kin}).

\begin{thm}\label{Hww}
Let $\beta: (0,\infty)\to (0,\infty)$
be a   function satisfying
\[
{\lim}_{t\to\infty}\, \beta(t)=0 \qquad \text{and}\qquad
 t\beta(t) \ge 1 \qquad t \ge 1.
\]
Then there exists a finite (discrete) Borel measure $\nu=\nu_\beta$ on $(0,\infty)$ such that
\begin{equation}\label{zv}
\int_{(0,\infty)}\,t \beta(t) \nu(dt)<\infty,
\end{equation}
and the corresponding  $p$-function
$g \sim (0,\nu)$ has unbounded variation on $[0,\infty)$.
\end{thm}

\begin{proof}
By Theorem
\ref{betakN}, setting $w=(k \beta(k))_{k=1}^{\infty},$ there exists
 $f\in \mathcal{A}^{+}(w)$ such that
 \begin{equation}\label{under}
\underline{\lim}_{\theta\to 0+}\,\frac{|1-f(\theta)|}{\theta}=0.
\end{equation}
If $F(e^{i\theta})=f(\theta),$ then $z \to F(e^{-z})$ is of the form \eqref{Form} for an appropriate discrete measure $\nu$ supported by $\mathbb N.$ If a $p$-function $g$  is defined by
\[
\int_0^\infty e^{-zt} g(t)\, dt=
\frac{1}{z+1-F(z)},\qquad {\rm Re} \ z>0,
\]
then $g$  satisfies (\ref{zv}), and, moreover, it has unbounded variation on $[0,\infty)$ by
(\ref{lmC}) and (\ref{under}).

\end{proof}

\section{Appendix A: Technicalities}

\emph{The proof of Lemma \ref{cosAr}}\,\,
Since  $\lambda/2\in (0,\pi/2)$, $\gamma/2\in (0,\pi/2)$
and $(\gamma+\lambda)/2\in (0,\pi/3)$, we have
\[
\frac{\lambda}{4\gamma}\le
\frac{\lambda}{\pi\gamma}
\le
d_{\lambda,\gamma}\le
\frac{\pi \lambda}{4\gamma}
\le \frac{\lambda}{\gamma}.
\]

Moreover, as
\begin{equation*}\label{El1}
|1-e^{ia}|^2=4\sin^2(a/2)
\end{equation*}
and
\begin{equation*}\label{El2}
{\rm Re}\,((1-e^{ia})(1-e^{ib}))
=-4\sin(a/2)\sin(b/2)\cos((b+a)/2),
\end{equation*}
for $a, b \in \mathbb R,$
we obtain
\begin{equation*}\label{ABC}
|(1-e^{i \lambda})+d(1-e^{-i\gamma})|^2
=A_\gamma\left(d-\frac{B_{\lambda,\gamma}}{A_\gamma}\right)^2+
D_{\lambda,\gamma},
\end{equation*}
where
\begin{align*}\label{AC}
A_\gamma:=&
4\sin^2(\gamma/2),\\
B_{\lambda,\gamma}:
=&-{\rm Re}\,(1-e^{i\lambda})(1-e^{i\gamma})\\
=& 4\sin(\gamma/2)\sin(\lambda/2)
\cos((\gamma+\lambda)/2),\\
C_\lambda:=&4\sin^2(\lambda/2),\\
D_{\lambda,\gamma}:=&C_\lambda-
\frac{B^2_{\lambda,\gamma}}{A_\gamma}.
\end{align*}
Therefore,
\begin{equation*}\label{B/A}
\frac{B_{\lambda,\gamma}}{A_\gamma}=d_{\lambda,\gamma},
\qquad
D_{\gamma,\lambda}
=4\sin^2(\lambda/2)\sin^2((\gamma+\lambda)/2),
\end{equation*}
and then
\[
|(1-e^{i \lambda})+d_{\lambda,\gamma}(1-e^{-i\gamma})|
=2\sin(\lambda/2)\sin((\gamma+\lambda)/2)
\le \frac{\lambda(\gamma+\lambda)}{2}.
\]
$\Box$

To prove Littlewood's result mentioned in Section \ref{lp}
with an explicit constant,
we prove first the next auxiliary estimate; see also Remark \ref{litrem}.

\begin{lemma}\label{LL3}
If $E\subset [a,a+2\pi]$ is a measurable set such that
\[
\int_{E} (1-\cos t)\,dt\le \epsilon
\]
for some $\epsilon>0,$ then
\begin{equation}\label{EpsA}
{\rm meas} \, (E) \le  (4\pi^2)^{1/3}\epsilon^{1/3}.
\end{equation}
\end{lemma}

\begin{proof}
By $2\pi$-periodicity of $\cos t,$
we may assume without loss of generality that
$E\subset [-\pi,\pi].$
Write $I=[-{\rm meas} (E)/2, {\rm meas} \, (E)/2]$ and
\[
E_{1}=E\cap I, \qquad
E_{2}=E\setminus E_{1},
\qquad \text{and} \qquad
\tilde{E}= I \setminus E_{1}.
\]

Noting that ${\rm meas} \, (\tilde{E})={\rm meas} \, (E_{2}),$ we have
\begin{align*}
\int_{E_{2}} (1-\cos t)\,dt\ge& {\rm meas}\,(E_{2}) \min_{t\in E_{2}}(1-\cos t)\\
\ge& {\rm meas} \, (\tilde{E}) \max_{t \in \tilde{E}}(1-\cos t)\\
\ge&
\int_{\tilde{E}} (1-\cos t)\,dt,
\end{align*}
hence
\begin{align*}
\epsilon\ge \int_{E} (1-\cos t)\,dt
\ge& \int_{E_{1}} (1-\cos t)\,dt
+\int_{\tilde{E}} (1-\cos t)\,dt\\
=&\int_{I} (1-\cos t)\,dt,
\end{align*}
so that, as
\[
t-\sin t \ge \frac{t^3}{\pi^2},\qquad t \in [0,\pi/2],
\]
we obtain
\[
\epsilon\ge\int_{I} (1-\cos t)\,dt
=\left. 2(t-\sin t)\right|_{t={\rm meas}\,(E)/2}
\ge\frac{1}{4\pi^2}({\rm meas} \, (E))^3,
\]
and the statement follows.
\end{proof}

\begin{cor}\label{corcor}
Let $\alpha\ge 1$, $\beta \in \R, \epsilon >0,$ and let
a measurable set $E\subset [a,a+2\pi], a \in \mathbb R,$ be such that
\[
\int_{E} (1-\cos (\alpha t+\beta))\,dt\le \epsilon.
\]
Then
\begin{equation}\label{EpsAC}
{\rm meas} \, (E) \le (16\pi^2)^{1/3}\epsilon^{1/3}.
\end{equation}
\end{cor}

\begin{proof}
We have
\[
\int_{E} (1-\cos (\alpha t+\beta))\,dt
=\frac{1}{\alpha}\int_{G} (1-\cos t)\,dt,
\]
where
$G=\alpha E+\beta\subset[\alpha a+\beta,\alpha a+\beta+2\pi\alpha]$ and ${\rm meas} \, (G)=\alpha \, {\rm meas}\, (E)$ $ \le 2\pi \alpha.$
Set $n=[\alpha+1],$  $h=2\pi\alpha/n,$ and write
\[
G_k=[h_k, h_k+h]\cap G,\qquad h_k=\alpha a+\beta+
(k-1)h,\quad k=1,\dots,n,
\]
so that
\[
G=\cup_{k=1}^n G_k,\qquad \text{where} \qquad G_k\cap G_j=\emptyset,\quad k\not=j, \quad 1 \le k \le n.
\]
By assumption,
\[
\int_{G_k} (1-\cos t)\,dt\le \alpha \epsilon_k,\qquad k=1,\dots,n,\qquad
\epsilon_1+\dots +\epsilon_n=\epsilon.
\]
Hence, by Lemma \ref{LL3} and convexity arguments,
\begin{align*}
{\rm meas} \, (E)= \alpha^{-1} \sum_{k=1}^n {\rm meas} \, (G_k) \le& (4\pi^2)^{1/3}\alpha^{-2/3}\sum_{k=1}^n \epsilon_k^{1/3}\\
\le& (4\pi^2)^{1/3} \left(\frac{n}{\alpha}\right)^{2/3}\epsilon^{1/3}\\
\le& (16\pi^2)^{1/3}\epsilon^{1/3}.
\end{align*}
\end{proof}

Thus, we have the following result,
formulated essentially in
\cite[Lemma]{Littl} with a constant $A$ instead of $(16\pi^2)^{1/3}$.

\begin{lemma}\label{Llemma}
Let $$f(t)=\sum_{m=1}^\infty a_m\cos(\alpha_m t+\beta_m)$$
where $\alpha_m\ge 1$, $\beta_m \in \mathbb R,$ $m\in \N$, and
\[
a_m\ge 0,\qquad \sum_{m=1}^\infty a_m=1.
\]
If $\epsilon>0 $ and  $E \subset [-\pi,\pi]$ is a measurable set such that
\[
\int_E (1-f(t))\,dt\le \epsilon,
\]
then
${\rm meas}\, (E)\le (16\pi^2)^{1/3}\epsilon^{1/3}$.
\end{lemma}

Now we are able to prove Theorem \ref{Lth} with the absolute constant $4\pi.$
\vspace{0,3cm}

{\it Proof of Theorem \ref{Lth}.}\,\,
Let $f$ be as in Theorem \ref{Lth}, and let
$E_\epsilon:=\{t \in[-\pi,\pi]: f(t)\ge 1-\epsilon \}.$
Then
\[
\int_{E_\epsilon} (1-f(t))\,dt\le \epsilon\, {\rm meas} \, (E_\epsilon),
\]
and from Lemma \ref{Llemma} it follows that
$
{\rm meas}\,(E_\epsilon)\le (16\pi^2)^{1/3}\epsilon^{1/3}({\rm meas}\,
( E_\epsilon))^{1/3},
$
hence
\[
{\rm meas} \, (E_\epsilon)\le (16\pi^2)^{1/2}\epsilon^{1/2}=
4\pi\epsilon^{1/2}.
\]
\hfill \hfill \hfill \hfill \hfill \hfill \hfill \hfill \hfill \hfill \hfill \hfill \hfill \hfill \hfill $\Box$

We now turn to the proof of Proposition \ref{Cake2}. First, let us recall one of the forms of
layer-cake representation, which is a direct consequence of  e.g.
\cite[Prop. 1.3.4]{Garling} or \cite[Ch. 1.13]{Lieb}.
\begin{lemma}\label{Cake1}
Let $\psi: [0,\infty) \mapsto (0,\infty)$  be a  differentiable strictly {\it decreasing} function
with $\lim_{t\to\infty}\,\psi(t)=0$,
and let
$\varphi: \Omega \mapsto (0,\infty)$ be a measurable function, where $\Omega\subset [0,\infty)$ is a measurable set of
finite measure. Then
\begin{equation}\label{ReprC}
\int_{\Omega} \psi(\varphi(s))\,ds=
-\int_0^\infty \psi'(t) g_{\varphi}(t)\,dt,
\end{equation}
where
\begin{equation}\label{gvarphi}
g_{\varphi}(t):= \{
s \in \Omega :\,\varphi(s) \le t\},\qquad t>0.
\end{equation}
\end{lemma}
(It suffices to note  that
\[
\int_{\Omega} \psi(\varphi(s))\,ds=
-\int_{\Omega} [\psi(0)-\psi(\varphi(s))]\,ds
+\psi(0){\rm meas}(\Omega)
\]
and apply either of the statements from \cite{Garling} or \cite{Lieb} to $\psi(0)-\psi(t)$.)

Having Lemma \ref{Cake1} in mind, we are now ready to give a proof of Proposition \ref{Cake2}.
\vspace{0,3cm}

{\it Proof of Proposition \ref{Cake2}}. \,\,
Let $g_{\varphi}$ be defined
by (\ref{gvarphi}) so that $$g_{\varphi}(t)={\rm meas} (\Omega) - {\rm meas}\,\{
s:\,\varphi(s)>t\}, \qquad t>0.$$  By assumption,
   \[
   g_{\varphi}(t)\le A\sqrt{\frac{t}{r}},\quad t\in [\eta,r].
   \]
Using Corollary
\ref{Cake1} with $\psi(t)=1/(t+d)^p,$
we obtain
\begin{equation}\label{cake3}
\int_{\Omega} \frac{ds}{(\varphi(s)+d)^p}= p\int_0^\infty \frac{g_\varphi(t)\,dt}{(t+d)^{p+1}}
=\frac{{\rm meas} \, (\Omega)}{(r+d)^p}+
 p\int_\eta^r \frac{g_\varphi(t)\,dt}{(t+d)^{p+1}}.
   \end{equation}
Next,  using the estimate (see e.g. \cite[Thm. 41]{Ineq})
\[
\frac{1}{x^q}-\frac{1}{y^q}
\le \frac{q(y-x)}{x^q y},\qquad y\ge x>0,\quad q\ge 1,
\]
and the elementary inequality
\[
(x+y)^q\ge x^q+y^q,\qquad x,y\ge 0,\quad q\ge 1,
\]
we have
   \begin{align*}
    \int_\eta^r \frac{g_\varphi(t)\,dt}{(t+d)^{p+1}}\le&
   \frac{A}{\sqrt{r}}\int_\eta^r \frac{\sqrt{t}\,dt}{(t+d)^{p+1}}\\
   \le&
   \frac{A}{\sqrt{r}}\int_\eta^r \frac{dt}{(t+d)^{p+1/2}}\\
   =& \frac{A}{(p-1/2)\sqrt{r}}\left[\frac{1}{(\eta+d)^{p-1/2}}-
   \frac{1}{(r+d)^{p-1/2}}\right]\\
   \le& \frac{A}{(p-1/2)\sqrt{r}}
   \frac{(p-1/2)(r-\eta)}{(\eta+d)^{p-1/2}(r+d)}\\
   \le& \frac{ A\sqrt{r}}{(\eta+d)^{p-1/2}(r+d)} \\
   \le& \frac{2 A}{(\eta+d)^{p-1/2}(\sqrt{r}+\sqrt{d})}\\
       \le&
   \frac{2A }{\eta^{p-1/2}\sqrt{r} +d^p}.
   \end{align*}
   From this and (\ref{cake3}) we obtain (\ref{cake4}).
\hfill \hfill $\Box$

We finish this section with the proof of auxiliary Proposition \ref{n2} on positive series.
\vspace{0,3cm}

\emph{Proof of Proposition \ref{n2}}.\,\,
Note that
\[
 2^{\alpha n} q_{2^n}\le 2^\alpha j^\alpha q_j,\qquad
2^{n-1}+1 \le j \le 2^n,\quad n\ge m,
\]
so
\[
 2^{\alpha n} q_{2^n}\le
 \frac{2^\alpha}{2^{n-1}} \sum_{j=2^{n-1}+1}^{2^n} j^\alpha q_j
\le  2^{1+\alpha}\sum_{j=2^{n-1}+1}^{2^n} j^{\alpha-1} q_j,
 \]
 and then
 \[
 \sum_{n=m}^\infty 2^{\alpha n} q_{2^n}
\le 2^{1+\alpha}\sum_{n=m}^\infty \sum_{j=2^{n-1}+1}^{2^n} j^{\alpha-1} q_j
=2^{1+\alpha}\sum_{n=2^{m-1}+1}^\infty n^{\alpha-1} q_n.
\]

On the other hand,
\[
 2^{\alpha n} q_{2^n}\ge 2^{-\alpha} j^\alpha q_j,\qquad
2^n \le j \le 2^{n+1}-1,\quad n\ge m,
\]
hence
\[
 2^{\alpha n} q_{2^n}\ge \frac{1}{ 2^{\alpha}2^n} \sum_{j=2^n}^{2^{n+1}-1}j^\alpha q_j\ge
 \frac{1}{ 2^{\alpha}} \sum_{j=2^n}^{2^{n+1}-1}j^{\alpha-1} q_j,
\]
and
\[
\sum_{n=m}^\infty 2^{\alpha n} q_{2^n}
\ge
 \frac{1}{ 2^{\alpha}} \sum_{n=m}^\infty \sum_{j=2^n}^{2^{n+1}-1}j^{\alpha-1} q_j
 = \frac{1}{ 2^{\alpha}} \sum_{n=2^{m}}^\infty n^{\alpha-1} q_n.
\]
\hfill \hfill \hfill \hfill \hfill \hfill \hfill \hfill \hfill \hfill $\Box$
\section{Appendix B: Croft's approach}

In this section we present a different proof of Corollary \ref{corcos}
based on Croft's approach from
\cite{Croft} dealing with
 quasi-exponential series, that is with trigonometrical series with real frequencies.
However, as we remarked above, the argument in \cite{Croft} seems to be incomplete.

Let us briefly compare Croft's approach with the one of the present paper.
For a sufficiently small parameter $\theta_0$ both approaches aim at finding $d$ and $\gamma$ such that
$\delta(\theta_0)=|1-P_{d,\gamma}(\theta_0)|$
is ``small'', in particular,
$\delta(\theta_0)\le c \theta_0^{2-\epsilon}$ for some $\epsilon \in (0,1],$
where a constant $c$ does not depend on $\theta_0.$
Croft proceeds by requiring
${\rm Im}\, P_{d,\gamma}(\theta_0)=0.$
This way he expresses $\gamma$ in terms of $d$ and then chooses $d$ to ensure the inequality above.
Unfortunately, his proof stops at this step.
Proceeding in a different way,  for a fixed $\gamma$ we minimize the quadratic function
$q(d)=(1+d)^2|1-P_{d,\gamma}(\theta_0)|^2$ with respect to $d.$
This relates $d$ to  $\gamma,$ and allows us to
make the quantity $\delta$ ``small'' enough to fit the steps of our inductive constructions in Section \ref{aux}.
The two steps lead eventually to similar estimates of $\delta(\theta_0).$
On the other hand, we also have to take care of a) getting polynomials with integer frequencies eventually, b) spreading out our estimates for $\delta$
to large sets, c) extending our estimates for polynomials $(Q_m)_{m=1}^{\infty}$
from fixed $\theta_0$ to appropriate sequences of  $(\theta_m)_{m=1}^{\infty}\subset [0,\pi]$ going to $0$ and then, finally, d)  of
constructing $f \in \mathcal A^{+}(w)$ out of  $(Q_m)_{m=1}^{\infty}$ via a limiting procedure.

\begin{lemma}\label{Pol}
Let
\[
P(\theta)=\sum_{k=1}^n a_k e^{ik\theta}\in \mathcal{A}^{+}
\]
$\epsilon\in (0,1],$
and $\alpha\in (0,\epsilon/2].$
Then for any
$\theta_0\in (0,1/n^{1/\alpha}],$
\begin{equation}\label{ImP}
0<{\rm Im}\, P(\theta_0)
\le\theta_0^{1-\alpha}<1,
\end{equation}
and
\begin{equation}\label{ReP}
0\le  1-{\rm Re}\,P(\theta_0)\le\theta_0^{2-\epsilon}.
\end{equation}
Moreover,
\begin{equation}\label{MoreP}
\|P'\|_{\infty} \le \theta_0^{-\alpha}.
\end{equation}
\end{lemma}

\begin{proof}
Since
$n\theta_0\in (0,1]$ and $n\le \theta_0^{-\alpha}$,
\[
0<\frac{2\theta_0}{\pi}\sum_{k=1}^n k a_k
\le {\rm Im}\, P(\theta_0)
=\sum_{k=1}^n a_k\sin(k\theta_0)
\le n\theta_0\le\theta_0^{1-\alpha}<1,
\]
and, in view of $\epsilon\ge 2\alpha$,
\[
0\le  1-{\rm Re}\,P(\theta_0)=\sum_{k=1}^n a_k(1-\cos(k\theta_0))\le n^2\theta_0^2/2
\le\theta_0^{2(1-\alpha)}\le\theta_0^{2-\epsilon},
\]
so that (\ref{ImP}) and (\ref{ReP}) hold.
Moreover,
\[
|P'(\theta)|\le
\sum_{k=1}^n k a_k\le n\le \theta_0^{-\alpha},\qquad |\theta|\le \pi,
\]
and we get (\ref{MoreP}).
\end{proof}

As an illustration, to show that Croft's idea actually works,
we provide now a proof of Corollary
\ref{corcos} following Croft's approach.

{\it Proof of Corollary
\ref{corcos}.} \,\,
 For $d>0$ and $\tau>0$, define
\begin{equation}\label{DefAA}
Q(\theta)=\frac{P(\theta)+de^{i\tau\theta}}{1+d}.
\end{equation}

Then
\begin{align*}\label{AdMM}
1-Q(\theta)=\frac{(1-P(\theta))+d(1-e^{i\tau\theta})}{1+d}
\end{align*}
and
\begin{align*}
P(\theta)-Q(\theta)=
\frac{d(P(\theta)-e^{i\tau\theta})}{1+d}.
\end{align*}

Fix $\Theta_0\in (0,1)$ such that
\begin{equation}\label{withoutA}
\frac{n}{\varphi^{1/4}(\theta)}\le 1,\qquad \text{and}\qquad
\nu+\frac{\log\varphi(\theta)}{|\log\theta|}\le 1,\qquad
\theta\in (0,\Theta_0],
\end{equation}
so that, in particular,
\begin{equation}\label{gwA}
\varphi(\theta)\ge 1 \qquad \text{and} \qquad
\theta^{1-\nu}\varphi(\theta)\le 1,\quad
\theta\in (0,\Theta_0].
\end{equation}
For $\theta\in (0,\Theta_0],$ define
\[
\epsilon(\theta):=
\nu+\frac{\log\varphi(\theta)}{|\log\theta|}\in (0,1],
\qquad\alpha(\theta):=\frac{\log\varphi(\theta)}{4|\log\theta|}<
\frac{\epsilon}{4},
\]
and note that
\[
\theta^{\epsilon(\theta)}=\frac{\theta^\nu}{\varphi(\theta)},\qquad
\theta^{\alpha(\theta)}
=\frac{1}{\varphi^{1/4}(\theta)}.
\]
Then the assumption $\theta \in (0,1/n^{1/\alpha(\theta)}]$
takes the form
$n\le \varphi^{1/4}(\theta)$
and holds for $\theta\in (0,\Theta_0]$, by
(\ref{withoutA}).

Furthermore, for  $\theta\in (0,\Theta_0]$ define
\[
d(\theta):=
\frac{\theta^\nu}{\phi^{1/2}(\theta)} \qquad \text{and} \qquad
v(\theta):=\frac{{\rm Im}\,P(\theta)}{d(\theta)}.
\]
By (\ref{ImP}) and (\ref{gwA}) we infer that
\begin{equation}\label{Zv}
0<v(\theta)\le \frac{\theta^{1-\alpha(\theta)}}{d(\theta)}
=\theta^{1-\nu}\phi^{3/4}(\theta)
\le \theta^{1-\nu}\phi(\theta)
\le 1,
\end{equation}
hence
\begin{equation}\label{Bytheta}
\tau(\theta)=
\frac{2\pi-\arcsin v(\theta)}{\theta}, \qquad \theta \in (0,\Theta_0],
\end{equation}
is well-defined.
Observe that
$\tau$ is continuous on $(0,\Theta_0],$ and it satisfies
$\lim_{\theta\to 0+}\,\tau(\theta)=+\infty$. Therefore, there exists a sequence $(\theta_m)_{m \ge m_0}\subset (0,\Theta_0]$
such that
\begin{equation}\label{Admm}
\tau_m:=\tau(\theta_m)=m, \qquad m \ge m_0.
\end{equation}
We may assume that $\theta_{m+1}<\theta_m$. From
(\ref{Admm}) and (\ref{Bytheta}) it follows  that
$m \theta_m \le 2\pi, m \ge m_0.$
Moreover, the latter condition implies that
$m\theta_m\in (3\pi/2,2\pi]$ for  large $m$, so  without loss of generality, we may assume that
$
m\theta_m\in (3\pi/2,2\pi], m\ge m_0.
$

Define
a polynomial $Q_m$  by (\ref{DefAA}) with
$
d=d(\theta_m)$ and $\tau=\tau_m=m.
$
Then, employing  (\ref{Bytheta}) and (\ref{Admm}),
we obtain  that
\begin{equation}\label{impart}
(1+d(\theta_m)){\rm Im}\,Q_m(\theta_m)=
d(\theta_m)[v(\theta_m)+\sin(m\theta_m)]=0.
\end{equation}
Moreover, by (\ref{ReP}),
\begin{equation}\label{LiAB}
1-{\rm Re}\,Q_m(\theta_m)\le
\theta_m^{2-\epsilon(\theta_m)}+d(\theta_m)(1-\cos(m\theta_m)),
\end{equation}
and, in view of $m\theta_m\in (3\pi/2,2\pi], m \ge m_0,$  and  (\ref{Zv}),
\[
1-\cos(m\theta_m)\le
\sin^2(m\theta_m)
= v^2(\theta_m)
\le \theta_m^{2(1-\nu)}\phi^{3/2}(\theta_m).
\]
Thus, taking into account (\ref{LiAB}) and \eqref{impart},
\begin{align*}
|1-Q_m(\theta_m)|
\le \theta_m^{2-\epsilon(\theta_m)}+
\frac{\theta_m^\nu}{\phi^{1/2}(\theta_m)}
\theta_m^{2(1-\nu)}\phi^{3/2}(\theta_m)
=
2\varphi(\theta_m)\theta_m^{2-\nu}.
\end{align*}

Next, to estimate the distance between $P$ and $Q_m$ we note that
\[
\| P-Q_m\|_{\mathcal A}\le d(\theta_m)\frac{\| P-e^{i m\theta}\|_{\mathcal A}}{1+d(\theta_m)}
\le
2d(\theta_m)=
\frac{2\theta_m^\nu}{\varphi^{1/2}(\theta_m)},
\]
and then, by Remark \ref{Rem33},
\[
\| P-Q_m\|_{\mathcal A_\nu}\le m^\nu\|P-Q_m\|_{\mathcal A}
\le \frac{2(2\pi)^\nu}{\varphi^{1/2}(\theta_m)}.
\]

Finally,  (\ref{MoreP}), (\ref{withoutA}),
and (\ref{gwA}) imply that
\begin{align*}
\|Q_m'\|_\infty\le & \|P'\|_\infty + d(\theta_m) m\\
\le&
\theta_m^{-\alpha(\theta_m)}+
2\pi\frac{\theta_m^{\nu-1}}{\varphi^{1/2}(\theta_m)}
\\
=&\varphi^{1/4}(\theta_m)+2\pi\frac{\theta^{\nu-1}}{\varphi^{1/2}(\theta_m)}\\
\le& 8\frac{\theta_m^{\nu-1}}{\varphi^{1/2}(\theta_m)}.
\end{align*}
Therefore,
 if  $
|\theta_m-\theta|\le \varphi^{3/2}(\theta_m)\theta_m^{3-2\nu},
$
then
\begin{align*}
|1-Q_m(\theta)|
\le& 2\varphi(\theta_m)\theta_m^{2-\nu}+
8\frac{\theta_m^{\nu-1}}{\varphi^{1/2}(\theta_m)}
\varphi^{3/2}(\theta_m)\theta_m^{3-2\nu}\\
=&
10\varphi(\theta_m)\theta_m^{2-\nu}.
\end{align*}
\hfill \hfill \hfill \hfill \hfill \hfill \hfill \hfill$\Box$

\section{Acknowledgement}

The authors would like to thank J. Aaronson, who attracted the authors' attention to the problem posed by   Erd\H{o}s, de Bruijn, and Kingman.


\begin{thebibliography}{99}
\parskip=0.7pt

\bibitem{Aa} J. Aaronson, \emph{Rational weak mixing in infinite measure spaces,}
Ergodic Theory Dynam. Systems \textbf{33} (2013), 1611--1643.

\bibitem{Aa1} J. Aaronson, \emph{
Conditions for rational weak mixing,}
Stoch. Dyn. \textbf{16} (2016), 1660004, 12 pp.

\bibitem{AaB} J. Aaronson, \emph{An introduction to infinite ergodic theory,}
Mathematical Surveys and Monographs, \textbf{50}, AMS, Providence, RI, 1997.

\bibitem{Ball} K.~Ball and F.~Nazarov, \emph{Little level theorem and zero-Khinchin inequality for sums of independent random variables,} 1996, preprint, see http://users.math.msu.edu/users/fedja/prepr.html.

\bibitem{Bened} M. Benedicks, \emph{An estimate of the modulus of the characteristic function of a lattice distribution with application to remainder term estimates in local limit theorems,} Ann. Probab. \textbf{3} (1975), 162--165.

\bibitem{Bruijn1} N. G. de Bruijn and P. Erd\H{o}s, \emph{Some linear and some quadratic recursion formulas,} I and II, Nederl. Akad. Wetensch. Proc. (=Indag. Math.) \textbf{13} (1951), 374--382, and \textbf{14} (1952), 152--163.

\bibitem{Bruijn}
N. G. de Bruijn and P. Erd\H{o}s,
\emph{On a recursion formula and on some Tauberian theorems},
J. Res. Nat. Bur. Standards \textbf{50} (1953), 161--164.

\bibitem{Croft}
H. T. Croft, \emph{Note on a paper of J. E. Littlewood},
J. London Math. Soc. \textbf{37} (1962), 477--478.

\bibitem{Des} J.-M.  Deshouillers, G. A. Freiman, and A. A. Yudin
\emph{On bounds for the concentration function, I,
Structure theory of set addition,}
Ast\'erisque No. \textbf{258} (1999), 425--436.

\bibitem{Des2} J.-M.  Deshouillers, G. A. Freiman, and A. A. Yudin, \emph{On bounds for the concentration function, II,}
 J. Theoret. Probab. \textbf{14} (2001),  813--820.

\bibitem{Erdos}
P. Erd\H{o}s, W. Feller, and H. Pollard,
\emph{A property of power series with positive coefficients},
Bull. Amer. Math. Soc. \textbf{55} (1949), 201--204.

\bibitem{Ericsson}  K. B. Erickson, \emph{Strong renewal theorems with infinite mean,} Trans. Amer. Math. Soc. \textbf{151} (1970), 263--291.

\bibitem{Feller} W. Feller, \emph{An introduction to probability theory and its applications,} Vol. I,  II., Third ed., John Wiley and Sons, New York-London-Sydney, 1968-1971.



\bibitem{Garling}
D. J. H. Garling, \emph{Inequalities: a journey into linear analysis,}
Cambridge  University Press, 2007.

\bibitem{Giles} J. R. Giles, \emph{Introduction to the analysis of normed linear spaces,} Australian Math. Soc. Lect. Ser., \textbf{13}, Cambridge University Press, Cambridge, 2000.

\bibitem{MRL} A. Gomilko, M. Haase, and Yu. Tomilov, \emph{On rates in mean ergodic theorems,} Math. Res. Lett. \textbf{18} (2011), 201--213.

\bibitem{Hawkes}
J. Hawkes, \emph{A characteristic function and a problem of Kingman},
Bull. London Math. Soc. \textbf{9} (1977),  61--64.

\bibitem{Ineq}
G. H. Hardy, J. E. Littlewood and  G. P\'olya,
\emph{Inequalities},  Reprint of the 1952 ed.,  Cambridge University Press, Cambridge, 1988.

\bibitem{Ibragimov} I. A. Ibragimov, \emph{A remark on the ergodic theorem for Markov chains,} Teor. Verojatnost. i Primen. \textbf{20} (1975), 175--177 (in Russian), transl. in
Theory of Probability and its Appl. \textbf{20} (1975), 174--176.

\bibitem{Kendall}
D. G. Kendall, \emph{Renewal sequences and their arithmetic,}  Symposium on Probability Methods in Analysis (Loutraki, 1966),  Lecture Notes in Math. \textbf{31} (1967), Springer, Berlin, 147--175.Reprinted in: Stocahstic Analysis (Rollo Davidson Memorial Volume, ed. D. G. Kendall and E. F. Harding), Wiley, London, 1973, p. 47--72.

\bibitem{Kingman2} J. F. C. Kingman, \emph{The stochastic theory of regenerative events,}
Z. Wahrscheinlichkeitstheorie und Verw. Gebiete \textbf{2} (1964), 180--224.

\bibitem{Kingman1} J. F. C. Kingman, \emph{
Some further analytical results in the theory of regenerative events,}
J. Math. Anal. Appl. \textbf{11} (1965), 422--433.

\bibitem{Kingman}
J. F. C.  Kingman, \emph{Regenerative phenomena,}
Wiley Series in Probab. and Math. Stat., John Wiley, London-New York-Sydney, 1972.


\bibitem{Korevaar} J. Korevaar,
\emph{Tauberian theory. A century of developments,} Grundl. der Math. Wiss. \textbf{329}, Springer, Berlin, 2004.

\bibitem{Yudin}
 M. B. Khripunova and A. A. Yudin,
  \emph{An estimate for the concentration function for a class of additive functions,}  Mat. Zametki \textbf{82}
   (2007),  598--605 (in Russian); transl. in Math. Notes
   \textbf{82} (2007), 535--541.

\bibitem{Lieb} E. H. Lieb and M. Loss, \emph{Analysis,}  Second ed., Graduate Studies in Math. \textbf{14}, AMS, Providence, RI, 2001.

\bibitem{Littl}
J. E. Littlewood, \emph{A theorem on trigonometrical series with $\sum (|a_m|+|b_m|)$
convergent, with an application to a convergence problem},
J. London Math. Soc. \textbf{37} (1962), 252--255.

\bibitem{Littlewood}
J. E.  Littlewood, \emph{Littlewood's miscellany,} Ed.  B\'ela Bollob\'as,
Cambridge University Press, Cambridge, 1986.


\bibitem{Postnikov} A. G. Postnikov,  \emph{Tauberian theory and its applications,}
Trudy Mat. Inst. Steklov. \textbf{144} (1979)  (in Russian),
transl. in Proc. Steklov Inst. Math. 1980, no. 2.

\bibitem{YudinP}
L. P. Postnikova and A. A. Yudin,
\emph{On the concentration function}, Theory of Probability and its Applications, \textbf{22}
(1978),  362--366.

\bibitem{Schilling}  R. Schilling, R. Song, and Z. Vondra{\v{c}}ek,  \emph{Bernstein functions},
 de Gruyter Studies in Mathematics \textbf{37}, Second ed., Walter de Gruyter, Berlin, 2012.

\bibitem{Spitzer}
F. Spitzer, \emph{Principles of random walk,} Second ed., GTM, \textbf{34}, Springer, New York-Heidelberg, 1976.

\bibitem{Stone}  N. J. Stone, \emph{On characteristic functions and renewal theory,}  Trans. Amer.
Math. Soc, \textbf{120} (1965), 327--342.

\end{thebibliography}
\end{document}